\documentclass[11pt, reqno, en]{amsart}%
\usepackage[font=scriptsize]{caption}
\usepackage{amscd}
\usepackage[ocgcolorlinks]{hyperref}
\usepackage{aliascnt}
\newcommand{\mytheorem}[3]{%
  \newaliascnt{#1}{#3}
  \newtheorem{#1}[#1]{#2}
  \aliascntresetthe{#1}
  \expandafter\newcommand\csname #1autorefname\endcsname{#2}
}

\usepackage{mathrsfs,amssymb,amsthm, amsmath, mathtools}
\usepackage[%
a4paper,
margin=1in,
marginpar=.8in,
includeheadfoot%
]{geometry}
\usepackage{graphicx}
\usepackage{etoolbox}
\usepackage{ifxetex}
\usepackage{enumerate}
\usepackage{enumitem}
\makeatletter
\renewcommand*\env@cases[1][1.15]{%
  \let\@ifnextchar\new@ifnextchar
  \left\lbrace
  \def\arraystretch{#1}%
  \array{@{}l@{\quad}l@{}}%
}
\makeatother
\usepackage[dvipsnames]{xcolor}
\definecolor{darkgreen}{rgb}{0,0.5,0}
\definecolor{darkred}{rgb}{0.7,0,0}

\providetoggle{langen}
\makeatletter
\newif\ifuserdefn@presubmit
\ifuserdefn@presubmit
\usepackage[left, modulo, displaymath, mathlines]{lineno}
\linenumbers
\usepackage{seqsplit}
\usepackage[right]{showlabels}
\usepackage{xstring}
\renewcommand{\showlabelsetlabel}[1]{
  \noexpandarg%
  \parbox[t]{\marginparwidth}{\raggedright\upshape\small\color{blue}\expandafter\seqsplit\expandafter#1}
}

\AtBeginDocument{\newgeometry{
    vmargin={0in,0in},
    hmargin={0.8in,1.3in},
    includeheadfoot
  }
}
\fi
\makeatother

\iftoggle{langen}{
  \usepackage[british]{babel}
  \usepackage{lmodern}
  \usepackage[T1]{fontenc}
  \newtheoremstyle{scheader}%
  {}
  {}
  {\itshape}
  {}
  {\scshape}
  {.}
  {0.5em}
  {}
  \theoremstyle{scheader}
  \newtheorem{defn}{Definition}[section]
  \newtheorem{mthm}{Theorem}

  \mytheorem{mcor}{Corollary}{mthm}
  \mytheorem{mlem}{Lemma}{mthm}

  \mytheorem{thm}{Theorem}{defn}
  \mytheorem{lem}{Lemma}{defn}
  \mytheorem{cor}{Corollary}{defn}
  \mytheorem{prop}{Proposition}{defn}
  
  \theoremstyle{remark}
  \newtheorem*{rmk}{Remark}
  \newtheorem{nrmk}{Remark}
  \newtheorem*{claim}{Claim}
  \newtheorem{nclaim}{Claim} 
  
  \newcommand{\I}{\mathrm{I}}

  }{
  \newtheorem{thm}{定理}[section]
  \mytheorem{defn}{定义}{thm}
  \mytheorem{lem}{引理}{thm}
  \mytheorem{cor}{推论}{thm}
  \mytheorem{prop}{命题}{thm}
  \newtheorem*{rmk}{注记}

  \renewcommand{\emph}[1]{\textbf{\texttt{#1}}}
  \newcommand{\I}{\mathrm{I}}
}
\AtBeginEnvironment{proof}{%
  \setcounter{stepnum}{0}
  \setcounter{nclaim}{0}
}
\newcounter{stepnum}

\def\XXint#1#2#3{{\setbox0=\hbox{$#1{#2#3}{\int}$ }
\vcenter{\hbox{$#2#3$ }}\kern-.6\wd0}}

\newcommand{\xtext}[1]{\ensuremath{\,\text{ #1 }}}

\newcommand{\II}{\mathrm{II}}
\newcommand{\III}{\mathrm{III}}

\newcommand{\myquad}[1][1]{\hspace*{#1em}\ignorespaces}
\usepackage[backrefs, msc-links, lite, abbrev, alphabetic]{amsrefs} 
\hypersetup{%
  pdfstartview=FitH,
  bookmarksopen=ture,
  pdfencoding=auto,
  colorlinks=false,
  linkcolor=cyan,
  pdfinfo={%
    Subject={58E20; 53C27; 53C50; 35J60; 35B65},
    Keywords={Dirac Harmonic maps, regularity, Lorentzian manifold, pseudo-Riemannian manifold}%
  }
}
\usepackage[perpage]{footmisc}

\usepackage{centernot}
\numberwithin{equation}{section}
\newcommand{\sepstar}{\par\vspace{3em}\hbox to \hsize{\hrulefill\qquad\lower1ex\hbox{*\qquad*\qquad*}\qquad\hrulefill}\vspace{3em}\par}

\newcommand{\D}{\centernot{D}}
\newcommand{\p}{\centernot{\partial}}
\let\div\relax
\DeclareMathOperator{\div}{\operatorname{div}}
\DeclareMathOperator{\curl}{\operatorname{curl}}
\title[Regularity for Dirac-Harmonic maps]{Regularity for Dirac-harmonic maps into certain pseudo-Riemannian manifolds}
\author{Wanjun Ai}
\address{School of Mathematics and Statistics\\
  Southwest University\\
  Chongqing 400715\\
P. R. China}
\address{School of Mathematical Sciences\\ %
  Shanghai Jiao Tong University\\ %
  Shanghai 200240\\ %
P. R. China}%
\email{wanjunai@swu.edu.cn}

\author{Miaomiao Zhu}
\address{School of Mathematical Sciences\\ %
  Shanghai Jiao Tong University\\ %
  Shanghai 200240 \\%
P. R. China}%
\email{mizhu@sjtu.edu.cn}%

\thanks{Part of this work was carried out when Wanjun Ai was a postdoc at the School of Mathematical Sciences, Shanghai Jiao Tong University and he would like to thank the institution for hospitality and financial support. Wanjun Ai is partially supported by the Fundamental Research Funds for the Central Universities (Grant No. SWU5330500542).
We would like to thank the referee for careful comments and helpful suggestions in improving the presentation of the paper.}
\date{\today}

\begin{document}
\begin{abstract}
  We show the smoothness of weakly Dirac-harmonic maps from a closed spin Riemann surface into stationary Lorentzian manifolds, and obtain a regularity theorem for a class of critical elliptic systems without anti-symmetry structures.
\end{abstract}
\maketitle
\section{Introduction}\label{sec:intro}
Motivated by the supersymmetric nonlinear sigma model from quantum field theory, e.g.\ \citelist{\cite{Deligne1999Quantum} \cite{Jost2009Geometry}}, the notion of Dirac-harmonic maps from spin Riemann surfaces into Riemannian manifolds were introduced in \cite{ChenJostLiWang2006Diracharmonic}. In the viewpoint of mathematics, they are generalizations of the classical harmonic maps and harmonic spinors. The action functional for Dirac-harmonic maps from spin Riemann surfaces preserves the conformal invariance, which makes the variational problem borderline cases of the Palais-Smale condition, and hence standard PDE methods can not be applied to get the regularity of critical points.

From the perspectives of sigma model from quantum field theory, see e.g.\ \cite{AlbertssonLindstromZabzine2003supersymmetric},  it is natural and of great interest to consider Dirac-harmonic maps from spin Riemann surfaces into pseudo-Riemannian manifolds, in particular, certain Lorentzian manifolds arising from general relativity e.g.\ \cite{KramerStephaniHerltMacCallum1980Exact, ONeill1983SemiRiemannian}. In this paper, we shall address this issue.  Suppose $(M^2,g_M)$ is a smooth and closed spin Riemann surface, $\Sigma M$ is a spinor (vector) bundle over $M$. A stationary Lorentzian manifold is a product manifold $\mathcal{N}\mathpunct{:}=\mathbb{R}^1\times N$, where $(N^n,g_N)$ is a compact Riemannian manifold of class $C^3$, equipped with a Lorentzian metric
\begin{equation}\label{eq:metric}
  g_{\mathcal{N}}=-\lambda(dr+\vartheta)^2+g_N,
\end{equation}
where $\lambda$ is a positive $C^2$ function on $N$; $\vartheta$ is a $C^2$ 1-form on $N$ and $dr^2$ is the standard metric on $\mathbb{R}^1$. Consider the space of smooth pairs $(\phi,\psi)$ defined by
\[
  \mathcal{X}(M,\mathcal{N})\mathpunct{:}=\left\{ (\phi,\psi)\mathpunct{:}\phi\in C^\infty(M,\mathcal{N})\text{ and } \psi\in \Gamma(\Sigma M\otimes\phi^{-1}T \mathcal{N}) \right\},
\]
and the following Lagrangian over $\mathcal{X}(M,\mathcal{N})$,
\begin{equation}\label{eq:L}
  \mathcal{L}(\phi,\psi)\mathpunct{:}=\frac{1}{2}\int_M g_{\mathcal{N}}(d\phi(e_\alpha),d\phi(e_\alpha))+\frac{1}{2}\int_M\left\langle \psi,\D\psi \right\rangle_{\Sigma M\otimes\phi^{-1}T \mathcal{N}},
\end{equation}
where $\left\{ e_\alpha \right\}$ is an orthonormal frame of $M$, $\left\langle \cdot,\cdot \right\rangle_{\Sigma M\otimes\phi^{-1}T \mathcal{N}}$ denotes the inner product induced from those on $\Sigma M$ and the pullback bundle $\phi^{-1}T \mathcal{N}$, and $\D$ is the \emph{Dirac operator along the map $\phi$}. Critical points $(\phi,\psi)\in \mathcal{X}(M,\mathcal{N})$ of \eqref{eq:L} are called \emph{Dirac-harmonic maps} from $M$ to $\mathcal{N}$.

In this paper, we shall investigate the regularity issue for Dirac-harmonic maps from a closed spin Riemann surface into stationary Lorentzian manifolds.
In order to define the weak solutions, we shall isometrically embed $(N,g_N)$ into some Euclidean space $(\mathbb{R}^K, g_0)$ and set
\[
  \bar{\iota} \mathpunct{:}=\mathrm{id}\times\iota: \mathcal{N}\longrightarrow \mathbb{R}^1\times \mathbb{R}^K\cong \mathbb{R}^{K+1}.
\]
The admissible space of weakly Dirac-harmonic maps is defined by
\begin{equation}\label{eq:weak-space}
  \begin{multlined}
    \mathcal{X}^w(M,\mathcal{N})\mathpunct{:}=\Bigl\{
      (\phi,\psi)\in W^{1,2}(M,\mathbb{R}^{K+1})\times W^{1,4/3}(M, \Sigma M\otimes\mathbb{R}^{K+1})\mathpunct{:}\text{ for a.e. $x\in M$, }\\
    \phi(x)\in \mathcal{N}\,\text{and for any $\nu\in T_{\phi(x)}^\perp \mathcal{N}$, $\left\langle \nu,\psi \right\rangle_{\mathbb{R}^{K+1}}=0$}\Bigr\},
  \end{multlined}
\end{equation}
where $\left\langle \cdot,\cdot \right\rangle_{\mathbb{R}^{K+1}}$ denotes the inner product induced from the pseudo-Riemannian metric of $\mathbb{R}^{K+1}$ (see \autoref{prop:iso-embedding} for the construction), which turns $\bar{\iota}$ into an isometrical embedding between pseudo-Riemannian manifolds. It is clear that $\mathcal{L}$ extends to the space $\mathcal{X}^w(M,\mathcal{N})$.
\begin{defn}\label{defn:DHM}
  A pair of fields $(\phi,\psi)\in \mathcal{X}^w(M,\mathcal{N})$ is called a \emph{weakly Dirac-harmonic map} from $M$ to $\mathcal{N}$, if it is a critical point of \eqref{eq:L} on $\mathcal{X}^w(M,\mathcal{N})$.
\end{defn}
Our main result reads:
\begin{mthm}\label{mthm:regularity}
  Suppose $M$ is a closed smooth spin Riemann surface, and $\mathcal{N}\mathpunct{:}=\mathbb{R}^1\times N$ is a $C^3$ stationary Lorentzian manifold with a $C^2$ metric given by \eqref{eq:metric}. If $(\phi,\psi)\in \mathcal{X}^w(M,\mathcal{N})$ is a weakly Dirac-harmonic map, then $\phi$ is H\"older continuous.
\end{mthm}
If, in addition,  $(\mathcal{N},g_{\mathcal{N}})$ is smooth, then we can improve the regularity and show the smoothness of weakly Dirac-harmonic maps.
\begin{mthm}\label{mthm:smooth}
  Suppose $M$ and $\mathcal{X}^w(M,\mathcal{N})$ are given as before, and $(\mathcal{N},g_{\mathcal{N}})$ is a smooth pseudo-Riemannian manifold. If $(\phi,\psi)\in \mathcal{X}^w(M,\mathcal{N})$ is a weakly Dirac-harmonic map and $\phi$ is continuous, then $(\phi,\psi)$ is smooth.
\end{mthm}

When the targets are spherical, a Jacobian structure for the weakly Dirac-harmonic maps was derived in \cite{ChenJostLiWang2005Regularity}*{Prop.~2.1}, and the regularity follows directly from Wente's lemma \cite{Wente1969existence}, see also H\'elein \cite{Helein2002Harmonic}*{Thm.~3.1.2}. When the targets are compact hyper-surfaces in $\mathbb{R}^n$, it was observed in \cite{Zhu2009Regularity} that the map part of a Dirac-harmonic map satisfies an elliptic system with an $L^2$-antisymmetric structure and hence the results by Rivi\`ere \cite{Riviere2007Conservation} and Rivi\`ere-Struwe \cite{RiviereStruwe2008Partial} can be applied to get the regularity for weak solutions. The case of general compact Riemannian targets was handled independently in \cite{WangXu2009Regularity} and in \cite{ChenJostWangZhu2013boundary}. See \citelist{\cite{ChenJostWangZhu2013boundary} \cite{SharpZhu2016Regularity}} for some boundary regularity results. For regularity theory of weakly harmonic maps from Riemann surfaces into compact Riemannian manifolds, we refer to \citelist{\cite{Helein2002Harmonic} \cite{Riviere2007Conservation} \cite{RiviereStruwe2008Partial}}.

When the target manifolds become non-compact or non-Riemannian, however, in general, the $L^2$-antisymmetric structure for harmonic map systems into compact Riemannian targets observed in \cite{Riviere2007Conservation} may not be preserved anymore. Therefore, it is of great interest to explore the extent to which the methods developed for elliptic systems with an $L^2$-antisymmetric structure can be generalized to elliptic systems of more general types. This is partially achieved in \cite{Zhu2013Regularity}, where the smoothness of weakly harmonic maps into certain pseudo-Riemannian manifolds, in particular, stationary Lorentzian manifolds, is proved by extending the results in \cite{Riviere2007Conservation, RiviereStruwe2008Partial} to certain critical elliptic systems without an $L^2$-antisymmetric structure. In \cite{Zhu2013Regularity}, it was shown that the harmonic map system into stationary Lorentzian manifolds can be written as a critical elliptic system with a potential which is a priori in $L^2$ but not necessarily antisymmetric, however, by exploring the geometric properties of the targets, it is observed that this potential has certain hidden antisymmetric structure and divergence free structure, which is crucial in proving the regularity.

In this paper, we shall extend the result in \cite{Zhu2013Regularity}*{Thm.~1.2} further by establishing a regularity theorem for a more general class of critical elliptic systems without an $L^2$-antisymmetric structure in general domain dimensions. Let $\Omega$ be a bounded domain in $\mathbb{R}^m$, $m\geq1$, recall that for $1\leq p<\infty$ and $\lambda\geq0$, the Morrey norm of a function $f\in L_{\mathrm{loc}}^p(\Omega)$, is defined as
\[
  \lVert f \rVert_{M_\lambda^p(\Omega)}\mathpunct{:}=\sup_{x\in\Omega, r>0}\left( \frac{1}{r^{m-\lambda}}\int_{B_r(x)\cap\Omega}\lvert f \rvert^p \right)^{1/p}.
\]

\begin{mthm}\label{mthm:reg-anti-protential}
  Suppose $B\subset \mathbb{R}^m$ is the unit ball, $m\geq2$ and $n>0$ are two integers. Denote $\mathrm{M}(n)$ to be the set of $n\times n$ real matrices. For any $\Lambda>0$, there exists $\epsilon=\epsilon(m,\Lambda)>0$, such that for every $\Theta\in L^2(B,\mathfrak{so}(n)\otimes\wedge^1 \mathbb{R}^m)$, $\Omega\in L^{2}(B, \mathrm{M}(n)\otimes\wedge^1\mathbb{R}^m)$, $F, G\in W^{1,2}\cap L^\infty(B, \mathrm{M}(n))$, $Q\in W^{1,2}\cap L^\infty(B,\mathrm{GL}(n))$ and $W\in M_2 ^q(B, \mathrm{M}(n))$, $\upsilon\in M_2 ^s(B,\mathbb{R}^n)$ for some $1<q<2$ and $1<s<2$. If $u\in W^{1,2}(B,\mathbb{R}^n)$ is a weak solution of the following elliptic system
  \begin{equation}\label{eq:thm-anti-u}
    -\div (Q\nabla u)=\Theta\cdot Q\nabla u+F\Omega\cdot G\nabla u+ \upsilon,
  \end{equation}
  where $\Omega$ satisfies
  \begin{equation}\label{eq:thm-anti-upsilon}
    -\div \Omega=W,
  \end{equation}
  with the coefficients satisfying the following conditions

  \begin{equation}\label{eq:thm-anti-condi}
    \begin{split}
    &\lVert \nabla u \rVert_{M_2 ^2(B)}+\lVert \Theta \rVert_{M_2 ^2(B)}+\lVert \Omega \rVert_{M_2 ^2(B)}+\lVert W \rVert_{M_2 ^q(B)}\\
    &\myquad[5]+\lVert \nabla Q \rVert_{M_2 ^2(B)}+\lVert \nabla F \rVert_{M_2 ^2(B)}+\lVert \nabla G \rVert_{M_2 ^2(B)}
    \leq\epsilon
    \end{split}
  \end{equation}
  and
  \begin{equation}\label{eq:thm-anti-condi-infty}
    \lvert Q \rvert+\lvert Q^{-1} \rvert+\lvert F \rvert+\lvert G \rvert\leq\Lambda,\quad \text{a.e. in $B$},
  \end{equation}
  then, for some $\alpha\in(0,1)$, we have
  \[
    [u]_{C^\alpha(B_{1/2})}\leq C(m,\Lambda,s)\left( \epsilon+\lVert \upsilon \rVert_{M_2 ^s(B)} \right).
  \]
  In particular, $u$ is H\"older continuous in $B_{1/2}$.
\end{mthm}

To prove \autoref{mthm:regularity}, we shall first derive the extrinsic version of the Euler-Lagrange system for the functional $\mathcal{L}$ by carefully exploring the extrinsic geometric structures of stationary Lorentzian manifolds, see \autoref{sec:EL}. Then, we shall rewrite the system for the map part into the same form as in \eqref{eq:thm-anti-u} and \eqref{eq:thm-anti-upsilon}, see \autoref{prop:weak-dirac-harmonic-local}. Finally, thanks to the improved regularity of the spinor part (see \autoref{sec:regularity-spinor}),  \autoref{mthm:reg-anti-protential} can be applied to get the H\"older continuity of the map part.

For some other analytic aspects of harmonic maps into pseudo-Riemannian manifolds, we refer to \citelist{\cite{Helein2004Removability}\cite{Zhu2013Regularity}}. For regularity of harmonic maps into static Lorentzian manifolds (taking the metric \eqref{eq:metric} with $\vartheta\equiv0$), see \cite{Isobe1998Regularity}.

The rest of the paper is organized as follows: in \autoref{sec:pre}, we set up the background and recall some basic properties of Dirac operator. The Euler--Lagrange equation for weakly Dirac-harmonic maps into stationary Lorentzian manifolds is derived in \autoref{sec:EL}, then we prove the continuity and smoothness of weakly Dirac-harmonic maps in \autoref{sec:conti} and \autoref{sec:smooth}, respectively. Finally, in \autoref{sec:appx}, we collect some analytic results needed for the proof of \autoref{mthm:reg-anti-protential}.

\section{Preliminaries}\label{sec:pre}

Suppose $(M^m,g_M)$ is a smooth and closed spin Riemannian manifold of dimension $m\geq2$, $\Sigma M$ is a spinor bundle over $M$ and $(\mathcal{N},g_{\mathcal{N}})$ is a pseudo-Riemannian manifold $\mathcal{N}$ of class $C^3$ with a pseudo-Riemannian metric $g_{\mathcal{N}}$. We will consider a pair of fields $(\phi,\psi)$, where $\phi$ is a map from $M$ to $\mathcal{N}$ and $\psi$ is a section of the twisted bundle $\Sigma M \otimes\phi^{-1}T\mathcal{N}$, $\psi$ is called a \emph{spinor field along the map $\phi$}. If $\phi$ is continuous, then in local coordinates, the section $\psi$ can be written as (we will always adopt the Einstein summation convention)
\[
  \psi=\psi^0\otimes\partial_r(\phi)+\psi^j\otimes\partial_{y^j}(\phi),
\]
where each $\psi^j$ is a usual spinor on $M$, and $\partial_r=\partial_{y^0}$, $\left\{ \partial_{y^j} \right\}_{j=1}^n$ are the natural local basis on $\mathbb{R}^1$ and $N$, respectively. Denote $\widetilde{\nabla}$ to be the Levi-Civita connection on $\Sigma M\otimes\phi^{-1}T\mathcal{N}$, induced from these on $\Sigma M$ and $\phi^{-1}T\mathcal{N}$ (see \cite{LawsonMichelsohn1989Spin}*{Thm.~4.17} and \cite{ONeill1983SemiRiemannian}*{Thm.~3.11, p.~61}). Locally,
\[
  \widetilde{\nabla}\psi=\nabla\psi^i\otimes\partial_{y^i}(\phi)+(\Gamma^i_{jk}(\phi)d\phi^j)\psi^k\otimes\partial_{y^i}(\phi),
\]
where $\left\{ \Gamma_{jk}^i \right\}_{i,j,k=0}^n$ are the Christoffel symbols of the \emph{Levi--Civita connection} of $(\mathcal{N}, g_{\mathcal{N}})$ (see, e.g.~\cite{ONeill1983SemiRiemannian}*{Defn.~3.12, p.~62}). The \emph{Dirac operator along the map $\phi$} is defined as
\begin{equation}\label{eq:Dirac-op}
  \D\psi\mathpunct{:}=e_\alpha\cdot\widetilde{\nabla}_{e_\alpha}\psi
  =\p\psi^i\otimes\partial_{y^i}(\phi)+\Gamma_{jk}^i(\phi)d\phi^j(e_\alpha)(e_\alpha\cdot\psi^k)\otimes\partial_{y^i}(\phi),
\end{equation}
where $\cdot$ is the Clifford multiplication from $\Gamma(TM)\times\Gamma(\Sigma M)$ to $\Gamma(\Sigma M)$, and $\p$ is the usual Dirac operator on $\Sigma M$, i.e., $\p\psi^i=e_\alpha\cdot\nabla^{\Sigma M}_{e_\alpha}\psi^i$.

Recall that there is a Hermitian product on $\Sigma M$ such that Clifford multiplication by the unit real vector is orthogonal (see, e.g., \cite{LawsonMichelsohn1989Spin}*{Chap.~I, Prop.~5.16}), the Riemannian metric induced from the Hermitian product is denoted by $\left\langle \cdot,\cdot \right\rangle_{\Sigma M}$, and we can require that the connection on $\Sigma M$ compatible with $\left\langle \cdot,\cdot \right\rangle_{\Sigma M}$. The metric of $\Sigma M\otimes\phi^{-1}T \mathcal{N}$ induced from these on $\Sigma M$ and $\phi^{-1}T \mathcal{N}$ is denoted by $\left\langle \cdot,\cdot \right\rangle_{\Sigma M\otimes\phi^{-1}T \mathcal{N}}$. When $M$ is closed, the Dirac operator $\D$ is formally self-adjoint (see, e.g., \cite{LawsonMichelsohn1989Spin}*{Chap.~II, Prop.~5.3}), i.e.,
\[
  \int_M\left\langle \psi_1,\D\psi_2 \right\rangle_{\Sigma M\otimes\phi^{-1}T \mathcal{N}}=\int_{M}\left\langle \D\psi_1,\psi_2 \right\rangle_{\Sigma M\otimes\phi^{-1}T \mathcal{N}},\quad\forall\psi_1,\psi_2\in\Gamma(\Sigma M\otimes{\phi}^{-1}T \mathcal{N}),
\]
where $\Gamma(\cdot)$ denotes the collection of smooth sections. For more details on spin geometry and semi-Riemannian geometry, we refer to \citelist{\cite{LawsonMichelsohn1989Spin}\cite{ONeill1983SemiRiemannian}}.

Let $\mathcal{X}(M,\mathcal{N})$ be the space of smooth pairs $(\phi,\psi)$ as defined in \autoref{sec:intro}. It is clear that the Lagrangian \eqref{eq:L} on $\mathcal{X}(M,\mathcal{N})$ is
\[
  \mathcal{L}(\phi,\psi)=\frac{1}{2}\int_M\left\langle \nabla\phi,\nabla\phi \right\rangle_{TM\times\phi^{-1}T\mathcal{N}} +\frac{1}{2}\int_M\left\langle \psi,\D\psi \right\rangle_{\Sigma M\times\phi^{-1}T \mathcal{N}}.
\]
By the non-degenerateness of $g_{\mathcal{N}}$, a direct computation as in \cite{ChenJostLiWang2006Diracharmonic}*{Prop.~2.1} shows that the Euler--Lagrange equations of $\mathcal{L}$ on $\mathcal{X}(M,\mathcal{N})$ are given by
\begin{equation}\label{eq:EL-smooth}
  \begin{cases}
    \tau(\phi)=\mathcal{R}(\psi,\nabla\phi\cdot\psi),\\
    \p\psi^i=-\Gamma_{jk}^i(\phi)\nabla\phi^j\cdot\psi^k,
  \end{cases}
\end{equation}
where $\tau(\phi)$ is the tension map of $\phi$, and locally
\begin{align*}
  \tau(\phi)&=\left( \Delta_M\phi^k+g_{M}^{\alpha\beta}\Gamma_{ij}^k(\phi)\frac{\partial\phi^i}{\partial x^\alpha}\frac{\partial \phi^j}{\partial x^\beta} \right)\partial_{y^k},\\
  \Gamma_{jk}^i&=\frac{1}{2}g_{\mathcal{N}}^{il}\left( \partial_{y^j}g_{\mathcal{N};lk}+\partial_{y^k}g_{\mathcal{N};jl}-\partial_{y^l}g_{\mathcal{N};jk} \right).
\end{align*}
$\mathcal{R}$ is defined by the pseudo-Riemannian curvature of $(\mathcal{N},g_{\mathcal{N}})$. More precisely,
\begin{equation}\label{eq:R-phi-psi}
  \mathcal{R}(\psi,\nabla\phi\cdot\psi)\mathpunct{:}=\mathcal{R}(\phi,\psi)\mathpunct{:}= \frac{1}{2}R_{ijl}^s(\phi)\left\langle \psi^i,\nabla\phi^l\cdot\psi^j \right\rangle_{\Sigma M}\partial_{y^s}(\phi),
\end{equation}
where $\left\{ R_{ijl}^s \right\}_{i,j,l,s=0}^n$ are the components of the pseudo-Riemannian curvature tensor $R$ of $(\mathcal{N},g_{\mathcal{N}})$, which is defined by $R(\partial_{y^i},\partial_{y^j})\partial_{y^k}=R^l_{ijk}\partial_{y^l}$. The index is lowered by the metric as $R_{ijkl}=R^s_{ijk}g_{\mathcal{N};sl}$. It has the same symmetries as Riemannian curvature tensor, see e.g., \cite{ONeill1983SemiRiemannian}*{Prop.~3.36, p.~75}.

In what follows, we turn to the extrinsic point of view by isometrically embedding $(N,g_N)$ to another Riemannian manifold $(\overline{N}, \bar{g})$ of dimension $K$, and the results will be applied to the case $\overline{N}=\mathbb{R}^K$ in \autoref{subsec:EL-extrinsic}. Firstly, we note the following proposition.
\begin{prop}\label{prop:iso-embedding}
  Suppose $(N,g_N)$, $(\overline{N},\bar{g})$ are two Riemannian manifolds of class $C^3$, $(N,g_N)$ is compact, and $\iota \mathpunct{:}N \hookrightarrow\overline{N}$ is an isometrical embedding. Let $\mathcal{N}\mathpunct{:}= \mathbb{R}^1\times N$ be a Lorentzian manifold equipped with a metric given by \eqref{eq:metric}, and $\pi$ be the $C^2$ nearest projection map from a tubular neighborhood $V_\delta N$ of $N\subset\overline{N}$ to $N$. Extend the pullback function $\pi^*\lambda=\lambda\circ\pi$ and 1-form $\pi^*\vartheta$ on $V_\delta N$ to $\overline{N}$ by cut-off, such that they are equal to $\pi^*\lambda$ and $\pi^*\vartheta$ on $V_{\delta/2}N$; while on $\overline{N}\setminus V_{\delta}N$, they are equal to $1$ and $0$, respectively. On $\overline{\mathcal{N}}\mathpunct{:}=\mathbb{R}^1\times\overline{N}$, if we define a $C^2$ pseudo-Riemannian metric $g_{\overline{\mathcal{N}}}$ as follows
  \begin{equation}\label{eq:g-bar-calN}
    g_{\overline{\mathcal{N}}}=-\pi^*\lambda\cdot(dr+\pi^*\vartheta)^2+\bar{g},
  \end{equation}
  then $\bar{\iota} \mathpunct{:}=\mathrm{id}\times\iota: \mathcal{N}\to\overline{\mathcal{N}}$ is an isometric embedding between pseudo-Riemannian manifolds.
\end{prop}
\begin{rmk}
  Geometrically, the above construction means $g_{\overline{\mathcal{N}}}$ is the standard Lorentzian metric $-dr^2+\bar{g}$ on $\overline{\mathcal{N}}\setminus (\mathbb{R}^1\times V_\delta N)$ and when restricted to $\mathcal{N}$, it is exactly the metric $g_{\mathcal{N}}$. In particular, since $N$ is compact, $g_{\overline{\mathcal{N}}}$ and all its derivatives are $L^\infty$ bounded on $\overline{\mathcal{N}}$, which implies that the pseudo-Riemannian Christoffel symbols and the pseudo-Riemannian curvature of $(\overline{\mathcal{N}}, g_{\overline{\mathcal{N}}})$ are $L^\infty$ bounded. Note that the second fundamental form of $\mathcal{N}\subset\overline{\mathcal{N}}$ is also $L^\infty$ bounded, see \eqref{eq:bdd-A}.
\end{rmk}
Thanks to the above proposition, we can view $\mathcal{N}$ as a submanifold of $\overline{\mathcal{N}}$, and define the second fundamental form  as in Riemannian case, i.e.,
\[
  \bar{A}(X,Y)\mathpunct{:}=\left( \nabla_X^{\bar{\iota}^{-1}T\overline{\mathcal{N}}}(d\bar{\iota}) \right)(Y)=\nabla_{d\bar{\iota}(X)}^{T\overline{\mathcal{N}}}(d\bar{\iota}(Y))-d\bar{\iota}\left( \nabla_X^{T \mathcal{N}}Y \right),\quad X,Y\in\Gamma(T \mathcal{N}).
\]
Let
\[
  \phi=(\varphi^0,\varphi),\quad \bar{\phi}=\bar{\iota}\circ\phi, \quad \bar{\varphi}=\iota\circ\varphi,\quad\bar{\psi}=\bar{\iota}_*\psi.
\]
Clearly, if $\psi$ is a spinor field along the map $\phi$, then $\bar{\psi}$ is a spinor field along the map $\bar{\phi}$.
Denote by $A$ the second fundamental of $\iota \mathpunct{:}N\hookrightarrow\bar{N}$, then the tension fields of $\varphi$ and $\bar{\varphi}$ are related by
\begin{equation}\label{eq:tau-relation}
  \tau(\bar{\varphi})=A(\nabla\varphi,\nabla\varphi)+\iota_*\bigl(\tau(\varphi)\bigr).
\end{equation}
If we denote $\bar{\D}$ to be the Dirac operator along the map $\bar{\phi}$, then
\begin{equation}\label{eq:bar-D}
  \bar{\D}\bar{\psi}=\bar{\iota}_*(\D\psi)+\bar{A}(d\phi(e_\alpha),e_\alpha\cdot\psi),
\end{equation}
where
\[
  \bar{A}(d\phi(e_\alpha),e_\alpha\cdot\psi)\mathpunct{:}=\nabla\phi^l\cdot\psi^j\otimes \bar{A}_{jl},\quad
  \bar{A}_{jl}\mathpunct{:}=\bar{A}(\partial_{y^j},\partial_{y^l}).
\]

Denote $\bar{R}$ and $R$ to be the pseudo-Riemannian curvature of $(\overline{\mathcal{N}},g_{\overline{\mathcal{N}}})$ and $(\mathcal{N},g_{\mathcal{N}})$, respectively. Define $\overline{\mathcal{R}}$ as the same as $\mathcal{R}$ in \eqref{eq:R-phi-psi}, except replacing $R$ with $\bar{R}$. By Gauss equation (see \cite{ONeill1983SemiRiemannian}*{Thm.~4.5, p.~100}) and the skew-adjointness relation of Clifford multiplication (see \cite{LawsonMichelsohn1989Spin}*{Chap.~I, Cor.~5.17}),
\begin{equation}\label{eq:bar-R}
  \mathcal{R}(\phi,\psi)=\overline{\mathcal{R}}(\phi,\psi)+\bar{P}(\bar{A}(d\phi(e_\alpha),e_\alpha\cdot\psi);\psi),
\end{equation}
where $\bar{P}(\bar{A}(d\phi(e_\alpha),e_\alpha\cdot\psi);\psi)$ is defined by the shape operator $\bar{P}$ (with abuse of notation) as follows:\[
  \bar{P}(\bar{A}(d\phi(e_\alpha),e_\alpha\cdot\psi);\psi)\mathpunct{:}=g_{\mathcal{N}}^{sk}\left\langle\bar{P}(\bar{A}_{jl};\partial_{y^i}), \partial_{y^k}\right\rangle_{T \mathcal{N}} \left\langle \psi^i,\nabla\phi^l\cdot\psi^j \right\rangle_{\Sigma M}\partial_{y^s}.
\]

Finally, we shall make a remark about the isometric embedding and the music isomorphism. Suppose $\bar{\iota} \mathpunct{:}(\mathcal{N}, g_{\mathcal{N}})\hookrightarrow(\overline{\mathcal{N}}, g_{\overline{\mathcal{N}}})$ is an isometric embedding between pseudo-Riemannian manifolds. For any 1-form $\bar{\omega}\in\Gamma(T^* \overline{\mathcal{N}})$, let $\omega=\bar{\iota}^*\bar{\omega}\in\Gamma(T^* \mathcal{N})$ be the pullback 1-form, and $\bar{\omega}^\sharp\in\Gamma(T\overline{\mathcal{N}})$, $\omega^\sharp\in\Gamma(T \mathcal{N})$ be the corresponding vector fields via music isomorphism. It is easy to show $w^\sharp=(\bar{\omega}^\sharp)^\top$, i.e., the tangential part of $\bar{\omega}^\sharp$ in $T \mathcal{N}$. Equivalently, the following diagram commutes:
\begin{equation}\label{eq:commute-diag}
  \begin{CD}
    \bar{\omega}\in\Gamma(T^*\overline{\mathcal{N}})@>\bar{\iota}^*>>\Gamma(T^*\mathcal{N})\ni\omega\\
    @V{}^\sharp VV@VV{}^\sharp V\\
    \bar{\omega}^\sharp\in\Gamma(T\overline{\mathcal{N}})@>\top>>\Gamma(T\mathcal{N})\ni\omega^\sharp.
  \end{CD}
\end{equation}
In fact, we will only need the Riemannian case of \eqref{eq:commute-diag} with $\iota \mathpunct{:}(N,g_{\mathcal{N}})\to(\mathbb{R}^K,g_0)$, where $g_0$ is the standard Euclidean metric over $\mathbb{R}^K$.

\section{The Euler--Lagrange equations}\label{sec:EL}

Following the scheme in \cite{Zhu2013Regularity}, instead of employing the Euler-Lagrange equation \eqref{eq:EL-smooth} directly, we need to separate the time and spacial components in the equation of $\phi$. We shall first compute the Euler--Lagrange equation of $\mathcal{L}(\phi,\psi)$ in the smooth category, which is the content of \autoref{subsec:EL-smooth}. Then, in \autoref{subsec:EL-extrinsic}, we employ the extrinsic point of view by embedding $N$ isometrically into $\mathbb{R}^K$, and rewrite the intrinsic equation into the extrinsic one, from which we can define weak solutions via integration by parts.

\subsection{The Euler--Lagrange equation in the smooth category}\label{subsec:EL-smooth}
The computation of Euler--Lagrange equation is kind of classical, see \cite{ChenJostLiWang2006Diracharmonic}*{Prop.~2.1}. However, in order to prove the regularity of weakly Dirac-harmonic maps into stationary Lorentzian manifolds, we need to write the equation of $\phi$ into equations of $\varphi^0$ and $\varphi$, i.e., separate the time and spacial components.

We begin by expressing our Lagrangian \eqref{eq:metric} in local coordinates. Suppose
\[
  (y^0,y')=(y^0,y^1,\ldots,y^n)
\]
are local coordinates on $\mathbb{R}^1\times N$. Locally, $\phi$ can be written as $\phi=(\varphi^0,\varphi)\in \mathbb{R}\times \mathbb{R}^n$ with $\varphi=(\varphi^1,\ldots,\varphi^n)$. Write $\vartheta=\sum_{i=1}^n\vartheta_idy^i$. It is easy to show,
\[
  d\phi(e_\alpha)\mathpunct{:}=\phi_*e_\alpha \mathpunct{:}=\phi^i_\alpha\partial_{y^i}=\varphi^0_\alpha\partial_{y^0}+\varphi^i_\alpha\partial_{y^i},
\]
and
\begin{align*}
  g_{\mathcal{N}}\left( d\phi(e_\alpha),d\phi(e_\alpha) \right)&=\sum_{i,j=0}^ng_{\mathcal{N};ij}\phi^i_\alpha\phi^j_\alpha\\
                                                               &=-\lambda(\varphi)\lvert dy^0(d\varphi^0(e_\alpha))+\vartheta(d\varphi(e_\alpha)) \rvert^2+g_N\left( d\varphi(e_\alpha),d\varphi(e_\alpha) \right),
\end{align*}
Therefore,
\begin{multline*}
  \mathcal{L}(\phi,\psi)=\frac{1}{2}\int_M-\lambda(\varphi)\lvert dy^0(d\varphi^0(e_\alpha))+\vartheta(d\varphi(e_\alpha)) \rvert^2+\left\langle d\varphi(e_\alpha),d\varphi(e_\alpha) \right\rangle_{TN}\\
  +\frac{1}{2}\int_M\left\langle \psi,\D\psi \right\rangle_{\Sigma M\otimes\phi^{-1}T \mathcal{N}}.
\end{multline*}

Now, we are ready to show the separated Euler--Lagrange equations for $\mathcal{L}$ over $\mathcal{X}(M,\mathcal{N})$. The computation is trivial but tedious, basically follows from a combination of \cite{ChenJostLiWang2006Diracharmonic}*{Prop.~2.1} and \cite{Zhu2013Regularity}*{Thm.~1.3}.
\begin{prop}\label{prop:EL}
  The Euler--Lagrange equations for $\mathcal{L}(\phi,\psi)$, $\phi=(\varphi^0,\varphi)$, $(\phi,\psi)\in \mathcal{X}(M,\mathcal{N})$, are
  \begin{align}
    \tau(\varphi)&=\mathcal{R}^\sharp(\phi,\psi)-\mathcal{H}^\sharp \label{eq:EL-phi},\\
    \div_M\left( V^\sharp\lambda(\varphi) \right)&=\mathcal{R}_0(\phi,\psi) \label{eq:EL-V},\\
    \D\psi&=0\label{eq:EL-psi},
  \end{align}
  where $\tau(\varphi)$ is the tension field of $\varphi \mathpunct{:}M\to N$, and
  \begin{equation}\label{eq:EL-notation}
    \begin{split}
      V^\sharp&=\left( dy^0(d\varphi^0(e_\alpha))+ \vartheta(\varphi)(d\varphi(e_\alpha)) \right)e_\alpha,\\
      \mathcal{H}^\sharp&=\sum_{j,k=1}^ng_{N}^{jk}\mathcal{H}_j\partial_{y^k}(\varphi),\\
      \mathcal{H}_j&=\frac{1}{2}\partial_j\lambda(\varphi)\lvert V^\sharp \rvert^2-\div_M(\lambda(\varphi)V^\sharp)\vartheta_j(\varphi)\\
                   &\qquad-\bigl( \partial_k\vartheta_j(\varphi)-\partial_j\vartheta_k(\varphi) \bigr)\left\langle \lambda(\varphi)V^\sharp,\nabla^M\varphi^k \right\rangle_{TM},\\
      \mathcal{R}_0(\phi,\psi)&\mathpunct{:}=\left\langle \mathcal{R}(\psi,\nabla\phi\cdot\psi),\partial_{y^0}(\phi) \right\rangle_{\phi^{-1}T \mathcal{N}}=\frac{1}{2}R_{ijl0}(\phi)\left\langle \psi^i,\nabla\phi^l\cdot\psi^j \right\rangle_{\Sigma M},\\
      \mathcal{R}^\sharp(\phi,\psi)&\mathpunct{:}=\frac{1}{2}\sum_{s,k=1}^ng_N^{ks}R_{ijlk}(\phi)\left\langle \psi^i,\nabla\phi^l\cdot\psi^j \right\rangle_{\Sigma M}\partial_{y^s},
    \end{split}
  \end{equation}
  $R$ is the pseudo-Riemannian curvature tensor of $(\mathcal{N}, g_{\mathcal{N}})$, and $g_N$ is the Riemannian metric of $N$.
\end{prop}
\begin{proof}
  Take a local orthonormal frame $\left\{ e_\alpha \right\}$  with $\nabla_{e_\alpha}e_\beta=0$ at $x\in M$, and note that $\nabla_{\frac{\partial}{\partial t}}\frac{\partial}{\partial t}=\nabla_{\frac{\partial}{\partial t}}e_\alpha=\nabla_{e_\alpha}\frac{\partial}{\partial t}=0$ locally. Suppose $\left\{ \psi_t \right\}$ is a family of variation with $d\psi_t/dt=\eta\in\Gamma(\Sigma M\otimes\phi^{-1}T \mathcal{N})$ at $t=0$ and $\phi$ is fixed, then
  \begin{align*}
    \left. \frac{d\mathcal{L}(\phi,\psi_t)}{dt} \right\rvert_{t=0}&=\int_M\left\langle \eta,\D\psi \right\rangle_{\Sigma M\otimes\phi^{-1}T \mathcal{N}}+\left\langle \psi,\D\eta \right\rangle_{\Sigma M\otimes\phi^{-1}T \mathcal{N}}\\
                                                                  &=2\int_M\left\langle \eta,\D\psi \right\rangle_{\Sigma M\otimes\phi^{-1}T \mathcal{N}},
  \end{align*}
  by the self-adjoint property of $\D$. Therefore, by the non-degenerateness of $g_{\mathcal{N}}$, we obtain \eqref{eq:EL-psi}.

  Next, we consider a variation $\left\{ \phi_t \right\}$ of $\phi$ such that $d\phi_t/dt=\xi=\xi^0+\xi'$ at $t=0$, $\xi^0(\phi)=\zeta^0\partial_{y^0}\circ\varphi^0$, $\xi'(\phi)=\sum_{j=1}^n\zeta^j\partial_{y^j}\circ\varphi$, and the coefficients $\psi^j$ in $\psi_t=\sum_{j=0}^n\psi^j\otimes\partial_{y^j}(\phi_t)$ are independent of $t$. It is easy to show,
  \begin{align*}
    \left. \frac{d\mathcal{L}(\phi_t,\psi_t)}{dt} \right\rvert_{t=0}&=-\frac{1}{2}\int_M \left. \frac{d}{dt} \right\rvert_{t=0}\bigl\{\lambda(\varphi_t)\lvert dy^0(d\varphi_{t}^0(e_\alpha))+\vartheta(d\varphi_{t}(e_\alpha)) \rvert^2\bigr\}\\
                                                                    &\qquad+\frac{1}{2}\int_M \left. \frac{d}{dt} \right\rvert_{t=0}\lvert \nabla\varphi \rvert^2+\frac{1}{2}\int_M\left. \frac{d}{dt} \right\rvert_{t=0}\left\langle \psi_t,\D\psi_t \right\rangle_{\Sigma M\otimes\phi_t^{-1}T \mathcal{N}}\\
                                                                    \mathpunct{:}=\I+\II+\III.
  \end{align*}
  The processing of $\II$ and $\III$ are similar to \cite{ChenJostLiWang2006Diracharmonic}*{Prop.~2.1}, while $\I$ needs to be handled carefully. In fact, $\II$ is the variation of classical Dirichlet energy of harmonic maps (into $N$), which is given by
  \[
    \II=\int_M\left. \left\langle d\varphi_t(e_\alpha),\nabla_{\frac{\partial\varphi_t}{\partial t}}d\varphi_t(e_\alpha) \right\rangle_{TN}\right\rvert_{t=0}
    =-\int_M\left\langle \tau(\varphi),\xi' \right\rangle_{TN},
  \]
  where $\tau(\varphi)$ is the tension field of $\varphi \mathpunct{:}M\to N$, which is defined as the trace of $\nabla^{\varphi^{-1}TN} d\varphi$, i.e.,
  \[
    \tau(\varphi)=\nabla^N_{d\varphi(e_\alpha)}d\varphi(e_\alpha)-d\varphi\left( \nabla^M_{e_i}e_i \right).
  \]

  For $\III$, we note first that, by \eqref{eq:EL-psi},
  \[
    \III\mathpunct{:}=\frac{1}{2}\int_M \left. \frac{d}{dt} \right\rvert_{t=0}\left\langle \psi,\D\psi \right\rangle_{\Sigma M\otimes\phi^{-1}T \mathcal{N}}
    =\frac{1}{2}\int_M\left\langle \psi,\left. \frac{D}{dt} \right\rvert_{t=0}\D\psi \right\rangle_{\Sigma M\otimes \phi^{-1}T \mathcal{N}}.
  \]
  Note that $[\phi_{t*}e_\alpha,\phi_{t*}\partial_t]=\phi_{t*}[e_\alpha,\partial_t]=0$, we have
  \[
    \frac{D}{dt}\D\psi=\D\left( \psi^i\otimes\left( \nabla_{\frac{\partial\phi_t}{\partial t}}\partial_{y^i} \right)\circ\phi_t \right)+e_\alpha\cdot\psi^i\otimes R\left( \frac{\partial\phi_t}{\partial t},\phi_{t*}e_\alpha \right)\partial_{y^i}\circ\phi_t,
  \]
  where $R$ is the pseudo-Riemannian curvature operator of $(\mathcal{N}, g_{\mathcal{N}})$. On account of the above formula of $\frac{D}{dt}\D\psi$, and by the self-adjoint of $\D$, apply \eqref{eq:EL-psi} again, we see that
  \[
    \III =\frac{1}{2}\int_M\left\langle \psi, e_\alpha\cdot\psi^j\otimes R\left( \xi,\phi_*e_\alpha \right)\partial_{y^j}\circ\phi \right\rangle_{\Sigma M\otimes \phi^{-1}T \mathcal{N}}.
  \]
  Now, since $\xi=\sum_{k=0}^n\zeta^k\partial_{y^k}\circ\phi$, $\phi_*e_\alpha=\phi_\alpha^l\partial_{y^l}$, we have,
  \begin{align*}
    \III&=\frac{1}{2}\int_M\left\langle \psi^i\otimes\partial_{y^i}\circ\phi,e_\alpha\cdot\psi^j\otimes\left( \zeta^k\phi^l_\alpha R^s_{klj}\partial_{y^s}\circ\phi \right) \right\rangle_{\Sigma M\otimes \phi^{-1}T \mathcal{N}}\\
        &=-\frac{1}{2}\int_{M}\left\langle \psi^i,\nabla\phi^l\cdot\psi^j \right\rangle_{\Sigma M}\zeta^0(\varphi^0)R_{ij0l}(\phi) -\frac{1}{2}\int_{M}\sum_{k=1}^n\left\langle \psi^i,\nabla\phi^l\cdot\psi^j \right\rangle_{\Sigma M}\zeta^k(\varphi)R_{ijkl}(\phi)\\
        &=\int_M\zeta^0(\varphi^0)\mathcal{R}_0(\psi,\nabla\phi\cdot\psi)+\left\langle \mathcal{R}^\sharp(\psi,\nabla\phi\cdot\psi),\xi' \right\rangle_{TN},
  \end{align*}
  where $\nabla\phi^l=\phi^l_\alpha e_\alpha$, and $\mathcal{R}_0(\psi,\nabla\phi\cdot\psi)$,  $\mathcal{R}^\sharp(\psi,\nabla\phi\cdot\psi)$ are given in \eqref{eq:EL-notation}. Note the degenerateness of $g_\mathcal{N}$ when restricted to $TN$, the orthogonal decomposition $T\mathcal{N}=TN\oplus T^\perp N$ as in Riemannian case not holds anymore, see e.g., \cite{ONeill1983SemiRiemannian}*{Lem.~2.23, p.~49}.

  To compute $\I$, we set $V=V_0$ and
  \[
    V_t(e_\alpha)\mathpunct{:}=\left\langle V^\sharp_t,e_\alpha \right\rangle=dy^0(\varphi_{t*}^0e_\alpha)+ \vartheta\bigl(\varphi_{t*}e_\alpha\bigr),
  \]
  then
  \begin{align*}
    \I&\mathpunct{:}=-\frac{1}{2}\int_M \left. \frac{d}{dt} \right\rvert_{t=0}\bigl\{\lambda(\varphi_t)\lvert dy^0(\varphi_{t*}^0e_\alpha)+\vartheta\left(\varphi_{t*}e_\alpha\right) \rvert^2\bigr\}\\
      &=-\frac{1}{2}\int_M d_N\lambda(\varphi)(\xi')\lvert V(e_\alpha) \rvert^2-\int_M\lambda(\varphi)V(e_\alpha)\left. \frac{d}{dt} \right\rvert_{t=0}V_t(e_\alpha).
  \end{align*}
  Note that $d_N\lambda(\varphi)(\xi')=(\partial_j\lambda\zeta^j)\circ\varphi$. A direct computation shows,
  \begin{align*}
    \lambda(\varphi)V(e_\alpha)\left. \frac{d}{dt} \right\rvert_{t=0}\left[ dy^0\left(\varphi^0_{t*}(e_\alpha)\right) \right]
    &=\left\langle \nabla^M(\zeta^0(\varphi^0)),\lambda(\varphi)V(e_\alpha)e_\alpha \right\rangle_{TM},\\
    \lambda(\varphi)V(e_\alpha)\left. \frac{d}{dt} \right\rvert_{t=0}\left[\vartheta(\varphi_{t*}e_\alpha)\right]
    &=\left\langle \lambda(\varphi)V^\sharp,\left( \partial_j\vartheta_i(\varphi)\zeta^j(\varphi)+\vartheta_j(\varphi)\partial_i\zeta^j(\varphi) \right)\nabla^M\varphi^i \right\rangle_{TM}.
  \end{align*}
  Now, integration by parts gives,
  \begin{align*}
    \I &=-\int_M \zeta^0(\varphi^0)\div (\lambda(\varphi)V^\sharp)\\
       &\qquad-\int_M \left( \frac{1}{2}\partial_j\lambda(\varphi)\lvert V \rvert^2-\div_M(\lambda(\varphi)V^\sharp)\vartheta_j(\varphi) \right)\zeta^j(\varphi)\\
       &\qquad-\int_M\bigl( \partial_j\vartheta_i(\varphi)-\partial_i\vartheta_j(\varphi) \bigr)\left\langle \lambda(\varphi)V^\sharp,\nabla^M\varphi^i \right\rangle_{TM} \zeta^j(\varphi)\\
       &=-\int_M \zeta^0(\varphi^0)\div (\lambda(\varphi)V^\sharp)-\int_M\left\langle \mathcal{H}^\sharp,\xi' \right\rangle_{TN},
  \end{align*}
  where $\mathcal{H}^\sharp$ is given in \eqref{eq:EL-notation}.

  In conclusion, we obtain
  \begin{align*}
    \left. \frac{d\mathcal{L}(\phi_t,\psi)}{dt} \right\rvert_{t=0}
    &=-\int_M \zeta^0(\varphi^0)\left( \div_M\left( V^\sharp\lambda(\varphi) \right)-\mathcal{R}_0(\phi,\psi) \right) \\
    &\qquad-\int_M\left\langle \tau(\varphi)+\mathcal{H}^\sharp-\mathcal{R}^\sharp(\phi,\psi),\xi' \right\rangle_{TN},
  \end{align*}
  from which we deduce the equations \eqref{eq:EL-phi}--\eqref{eq:EL-V}.
\end{proof}
\subsection{The weak Dirac-harmonic map equation}\label{subsec:EL-extrinsic}
In what follows, we will consider the isometric embedding $\iota\mathpunct{:}N\to \overline{N}=\mathbb{R}^K$, and transform the Euler-Lagrange equations
\eqref{eq:EL-phi}--\eqref{eq:EL-psi} into extrinsic view, from which we can define the weak sense of Dirac-harmonic equation.

Denote $\bar{\iota}=\mathrm{id}\times\iota: \mathcal{N}\to \mathbb{R}\times \mathbb{R}^K=\mathpunct{:}\overline{\mathcal{N}}$, and recall that we extended $\lambda$, $\vartheta$ to $\overline{N}$ via the nearest projection and cut-off function (see \autoref{prop:iso-embedding}), then we can write $\vartheta$ as $\vartheta=\left( \vartheta_1,\ldots,\vartheta_K \right)\in \mathbb{R}^K$ and $\varphi$ as $\varphi=\left( \varphi^1,\ldots,\varphi^K \right)$, where $\left\{ \vartheta_i \right\}_{i=1}^K$, $\lambda$ are $C^2$ functions on $\overline{N}$ and $\left\{ \varphi^i \right\}_{i=1}^K$ are $W^{1,2}$ functions on $M$.
Locally, if $\left\{ \partial_{v^a} \right\}_{a=0}^K$ is a natural basis of $\overline{\mathcal{N}}$ with $\partial_{v^0}=\partial_{y^0}$ to be a basis of $\mathbb{R}^1$, then $\bar{\D}$ can be expressed by the usual Dirac operator as follows
\begin{equation}\label{eq:bar-Dirac}
  \bar{\D}\bar{\psi}=\p\bar{\psi}+\bar{\Gamma}(d\bar{\phi}(e_\alpha),e_\alpha\cdot\bar{\psi}),
\end{equation}
where
\begin{equation}\label{eq:bar-Gamma}
  \bar{\Gamma}(d\bar{\phi}(e_\alpha),e_\alpha\cdot\bar{\psi})\mathpunct{:}=\nabla\bar{\phi}^a\cdot\bar{\psi}^b\bar{\Gamma}_{ab}^c(\bar{\phi})\otimes\partial_{v^c}\circ\bar{\phi},
\end{equation}
and $\left\{ \bar{\Gamma}_{ab}^c \right\}_{a,b,c=0}^K$ are the Christoffel symbols of $(\overline{\mathcal{N}},g_{\overline{\mathcal{N}}})$.
Thus, by \eqref{eq:bar-D}, the Dirac equation in the Euler--Lagrange equation \eqref{eq:EL-psi} is transformed to
\begin{equation}\label{eq:psi-extrinsic}
  \p\bar{\psi}=\bar{A}(d\phi(e_\alpha),e_\alpha\cdot\psi)-\bar{\Gamma}(d\bar{\phi}(e_\alpha),e_\alpha\cdot\bar{\psi}).
\end{equation}
In local coordinates, if we denote $B=(B^a_i)_{K\times n}$ to be the matrix with $B^a_{i}=\partial\iota^a/\partial y^i$, and $\bar{B}=(\bar{B}_i^a)_{(K+1)\times(n+1)}$, $\bar{B}_i^a=\partial\bar{\iota}^a/\partial y^i$, then
\[
  \bar{B}_i^a=\begin{cases}
    1,&i=0=a,\\
    0,&i=0,a\neq0\text{ or } i\neq0,a=0,\\
    B_i^a,&i\neq0, a\neq0,
  \end{cases}
\]
and
\[
  \begin{alignedat}{3}
    \bar{\psi}^a&=\bar{B}^a_{j}\psi^j,&\quad \partial_{y^i}&=\bar{B}^a_{i}\partial_{v^a},&\quad dv^a&=\bar{B}_i^a dy^i,\\
    \nabla\bar{\phi}^a&=\nabla\phi^i \bar{B}^a_{i},&\quad\vartheta_j&=\sum_{a=1}^K\vartheta_aB_j^a,&\quad
    \bar{A}_{ij}&\mathpunct{:}=\bar{A}(\partial_i,\partial_j)=\bar{B}_i^a\bar{B}_j^b\bar{A}_{ab}.
  \end{alignedat}
\]
It is easy to show,
\[
  \bar{A}(d\phi(e_\alpha),e_\alpha\cdot\psi)=\nabla\phi^i\cdot\psi^j\otimes\bar{A}_{ij}(\phi)
  =\nabla\bar{\phi}^a\cdot\bar{\psi}^b\otimes\bar{A}_{ab}(\bar{\phi})=\mathpunct{:}\bar{A}\bigl(d\bar{\phi}(e_\alpha),e_\alpha\cdot\bar{\psi}\bigr),
\]
where $\bar{A}_{ab}$ is a normal vector field along $\mathcal{N}$ defined as follows
\[
  \bar{A}_{ab}\mathpunct{:}=-\left\langle \nabla^{\overline{\mathcal{N}}}_{\partial_{v^a}}\bar{\nu}_l,\partial_{v^b} \right\rangle_{\overline{\mathcal{N}}}\bar{\nu}_l,
\]
here, $\left\{ \bar{\nu}_l \right\}_{l=n+1}^K$ is a local orthonormal frame of $T^\perp \mathcal{N}\subset T\overline{\mathcal{N}}$. In fact, if we write $\bar{\nu}_l \mathpunct{:}=\bar{v}_l^a\partial_{v^a}$, then by the compatibility of pull-back connection, we know that
\[
  \bar{A}_{ij}=\bar{A}\left(\partial_{y^i},\partial_{y^j}\right)
  =-\left\langle \nabla_{\partial_{y^i}}^{\bar{\iota}^{-1}T\overline{\mathcal{N}}}\bar{\nu}_l,d\bar{\iota}(\partial_{y^j}) \right\rangle_{\overline{\mathcal{N}}}\nu_l
  =\bar{B}_i^a\bar{B}_j^b\bar{A}_{ab},
\]
and
\begin{align*}
  \bar{A}(d\phi(e_\alpha),e_\alpha\cdot\psi)&=\nabla\phi^i\cdot\psi^j\otimes\bar{A}(\partial_{y_i},\partial_{y^j})\\
                                            &=\bar{B}_i^a\nabla\phi^i\cdot\bar{B}_j^b\psi^j\otimes\bar{A}_{ab}(\bar{\phi})\\
                                            &=\nabla\bar{\phi}^a\cdot\bar{\psi}^b\otimes\bar{A}_{ab}(\bar{\phi}).
\end{align*}
In components, we can write \eqref{eq:psi-extrinsic} as
\begin{equation}\label{eq:psi-extrinsic-compo}
  \p\bar{\psi}^c=\left( \bar{A}_{ab}^c(\bar{\phi})-\bar{\Gamma}_{ab}^c(\bar{\phi}) \right)\nabla\bar{\phi}^a\cdot\bar{\psi}^b,\quad \bar{A}_{ab}(\bar{\phi})=\bar{A}_{ab}^c(\bar{\phi})\partial_{v^c}.
\end{equation}

In order to show the boundedness of second fundamental form, we note first that
\begin{equation}\label{eq:bdd-A}
  \bar{A}_{ij}
  \mathpunct{:}=-\left\langle \nabla_{\partial_{y^i}}^{\bar{\iota}^{-1}T\overline{\mathcal{N}}} \bar{\nu}_l,d\bar{\iota}(\partial_{y^j}) \right\rangle_{\overline{\mathcal{N}}} \bar{\nu}_l
  =-\left( \bar{B}_j^d\frac{\partial \bar{v}_l^c}{\partial y^i}+\bar{B}_i^a \bar{B}_j^d \bar{v}_l^b\bar{\Gamma}_{ab}^c \right)g_{\overline{\mathcal{N}};cd}\bar{\nu}_l.
\end{equation}
Noting that $\bar{\nu}_l\in T^{\perp}\mathcal{N}$ is orthonormal, we know that $\delta_{kl}=g_{\overline{\mathcal{N}}}\left( \bar{\nu}_k,\bar{\nu_l} \right)=g_{\mathbb{R}^K}\left( \bar{\nu}_k,\bar{\nu}_l \right)=\bar{v}_k^a\bar{v}_l^b\delta_{ab}=\bar{v}_k^a\bar{v}_l^b$. In particular, we see that for any $l$ and any $a$, $|\bar{v}_l^a|\leq 1$. Moreover, by the construction of the metric of $\overline{\mathcal{N}}$, we can take $\bar{\nu}_l(t,y)$ to be independent of time $t$, i.e., it depends only on $y\in N$, which implies that $\lvert \partial \bar{v}_l^c/\partial y^i \rvert$ is bounded as well. Finally, since $N$ is compact, we know that $\iota$ and its derivatives are bounded. This shows that $\lvert \bar{A}_{ij} \rvert$ is bounded.

It is also easy to rewrite the equation of map $\bar{\phi}=(\varphi^0,\varphi^1,\ldots,\varphi^K)$. Firstly, by \eqref{eq:bar-R},
\begin{equation}\label{eq:bar-phi-extrinsic}
  -\div \left(\lambda(\varphi)\left( \nabla\varphi^0+\vartheta_a\nabla\varphi^a \right)\right)=\overline{\mathcal{R}}_0(\bar{\phi},\bar{\psi})+\overline{P}_0(\bar{\phi},\bar{\psi}),
\end{equation}
where
\begin{align*}
  \overline{\mathcal{R}}_0(\bar{\phi},\bar{\psi})&\mathpunct{:}=\left\langle \overline{\mathcal{R}}(\bar{\phi},\bar{\psi}),\partial_{v^0} \right\rangle_{\bar{\phi}^{-1}T\overline{\mathcal{N}}}=\frac{1}{2}\bar{R}_{abc0}(\bar{\phi})\left\langle \bar{\psi}^a,\nabla\bar{\phi}^c\cdot\bar{\psi}^b \right\rangle_{\Sigma M},\\
  \overline{P}_0(\bar{\phi},\bar{\psi})&\mathpunct{:}=\left\langle \overline{P}(\bar{A}(d\bar{\phi}(e_\alpha),e_\alpha\cdot\bar{\psi});\bar{\psi}),\partial_{v^0} \right\rangle_{\bar{\phi}^{-1}T\overline{\mathcal{N}}}\\
                                       &=\left\langle \bar{A}_{bc},\bar{A}_{a0} \right\rangle_{\bar{\phi}^{-1}T\overline{\mathcal{N}}}\left\langle \bar{\psi}^a,\nabla\bar{\phi}^c\cdot\bar{\psi}^b \right\rangle_{\Sigma M}.
\end{align*}

To rewrite the equation of $\varphi$, we need the Gauss equation of semi-Riemannian geometry (see \cite{ONeill1983SemiRiemannian}*{p.~100, Thm.4.5}), i.e.,
\[
  \bar{R}_{ijkl}=R_{ijkl}+\left\langle \bar{A}_{ik},\bar{A}_{jl} \right\rangle_{T\overline{\mathcal{N}}}-\left\langle \bar{A}_{il},\bar{A}_{jk} \right\rangle_{T\overline{\mathcal{N}}},
\]
and the skew adjointness of Clifford multiplication, i.e.,
\[
  \left\langle \bar{\psi}^i,\nabla\bar{\phi}^l\cdot\bar{\psi}^j \right\rangle_{\Sigma M}=-\left\langle \nabla\bar{\phi}^l\cdot\bar{\psi}^j,\bar{\psi}^i \right\rangle_{\Sigma M},
\]
where $\left\langle \cdot,\cdot \right\rangle_{\Sigma M}$ is the Riemannian metric of $\Sigma M$. By \eqref{eq:commute-diag}, the above two relation implies that
\[
  \mathcal{R}^\sharp(\phi,\psi)=\overline{\mathcal{R}}^{\sharp\top}(\bar{\phi},\bar{\psi})+\overline{P}^{\sharp\top}(\bar{\phi},\bar{\psi}),
\]
where
\begin{align*}
  \overline{\mathcal{R}}^\sharp(\bar{\phi},\bar{\psi})&\mathpunct{:}=\frac{1}{2}\sum_{d=1}^K\bar{R}_{abcd}\left\langle \bar{\psi}^a,\nabla\bar{\phi}^c\cdot\bar{\psi}^b \right\rangle_{\Sigma M}\partial_{v^d}\circ\bar{\phi},\\
  \overline{P}^\sharp(\bar{\phi},\bar{\psi})&\mathpunct{:}=\sum_{d=1}^K\left\langle \bar{A}_{ad},\bar{A}_{bc} \right\rangle_{T\overline{\mathcal{N}}}\left\langle \bar{\psi}^a,\nabla\bar{\phi}^c\cdot\bar{\psi}^b \right\rangle_{\Sigma M}\partial_{v^d}\circ\bar{\phi}.
\end{align*}
It is also easy to show
\[
  \mathcal{H}^\sharp=\overline{\mathcal{H}}^{\sharp\top}, \quad\overline{\mathcal{H}}^\sharp \mathpunct{:}=\sum_{a=1}^K\mathcal{H}_a\partial_{v^a}\circ\varphi,
\]
where
\[
  \mathcal{H}_a=\frac{1}{2}\partial_a\lambda(\varphi)\left\lvert \overline{V}^\sharp \right\rvert^2+\mathrm{div}\left(\lambda(\varphi)\overline{V}^\sharp\right)\vartheta_a -\left( \partial_b\vartheta_a-\partial_a\vartheta_b \right)\left\langle \lambda(\varphi)\overline{V}^\sharp,\nabla\varphi^b \right\rangle_{TM},
\]
and
\[
  V^\sharp\mathpunct{:}=\nabla\varphi^0+ \sum_{a=1}^K\vartheta_a\nabla\varphi^a.
\]
Finally, by \eqref{eq:tau-relation} and $\bar{N}=\mathbb{R}^K$,
\[
  \iota_*\left( \tau(\varphi) \right)=\tau(\bar{\varphi})-A(d\varphi(e_\alpha),d\varphi(e_\alpha))=\Delta_M\bar{\varphi}-A(\nabla\varphi,\nabla\varphi).
\]
We conclude the above discussion into the following proposition.
\begin{prop}\label{prop:EL-extrinsic}
  The Euler--Lagrange equation of $\mathcal{L}$ can be rewritten as follows (under the identification of $\bar{\varphi}=\iota\circ\varphi$ with $\varphi$, $\bar{\phi}$ with $\phi$ and $\bar{\psi}$ with $\psi$)
  \begin{equation}\label{eq:weak-Dirac-Harmonic}
    \begin{split}
      -\div \nabla\varphi&=A(\nabla\varphi,\nabla\varphi)+\overline{\mathcal{R}}^{\sharp\top}(\phi,\psi)+\overline{P}^{\sharp\top}(\phi,\psi)-\overline{\mathcal{H}}^{\sharp\top},\\
      -\div (\lambda(\varphi)V^\sharp)&=\overline{\mathcal{R}}_0(\phi,\psi)+\overline{P}_0(\phi,\psi),\\
      \p\psi&=\bar{A}(d\phi(e_\alpha),e_\alpha\cdot\psi)-\bar{\Gamma}(d\phi(e_\alpha),e_\alpha\cdot\psi),
    \end{split}
  \end{equation}
  where $\top$ is the tangential part of the Riemannian orthogonal decomposition $T \mathbb{R}^K=TN\oplus T^\perp N$.
\end{prop}
We have the following equivalent definition of weakly Dirac-harmonic maps. Recall that $\mathcal{X}^w$ is the admissible space of weakly Dirac-harmonic maps defined in \eqref{eq:weak-space}.
\begin{prop}\label{defn:weak-dirac-hm}
  A pair $(\phi,\psi)\in \mathcal{X}^w(M,\mathcal{N})$ is a \emph{weakly Dirac-harmonic map} from $M$ to $\mathcal{N}\mathpunct{:}=\mathbb{R}^1\times N$ if and only if for any $\eta^0\in C^\infty(M,\mathbb{R})$, $\eta\in C^\infty(M,\mathbb{R}^K)$ and any $\xi\in C^\infty\bigl(M,\Sigma M\otimes \mathbb{R}^{K+1})\bigr)$,
  \begin{align*}
    \int_M\left\langle \nabla\varphi,\nabla\eta \right\rangle-\left\langle A\left( \nabla\varphi,\nabla\varphi \right)+\overline{\mathcal{R}}^{\sharp\top}(\psi,\psi)-\overline{\mathcal{P}}^{\sharp\top}(\psi,\psi)-\overline{\mathcal{H}}^{\sharp\top},\eta \right\rangle&=0,\\
    \int_M\left\langle (\nabla\varphi^0+\vartheta_a\nabla\varphi^a)\lambda(\varphi),\nabla\eta^0 \right\rangle-\left( \overline{\mathcal{R}}_0(\phi,\psi)+\overline{P}_0(\phi,\psi) \right)\eta^0&=0,\\
    \int_M\left\langle \psi,\p\xi \right\rangle-\left\langle \bar{A}\left(d\phi(e_\alpha),e_\alpha\cdot\psi\right)-\bar{\Gamma}\left(d\phi(e_\alpha),e_\alpha\cdot\psi\right),\xi \right\rangle&=0,
  \end{align*}
  where the metric on $TM\otimes \mathbb{R}^K$ is the standard product metric of $g_M$ and Euclidean metric on $\mathbb{R}^K$, but the metric on $\Sigma M\otimes \mathbb{R}^{K+1}$ is the product metric of $\Sigma M$  and pseudo-Riemannian metric defined by \eqref{eq:g-bar-calN}.
\end{prop}

\

\section{The continuity of weakly Dirac-harmonic maps}\label{sec:conti}

Here and in the sequel, we will consider the regularity of weakly Dirac-harmonic maps from a closed Riemann surface $(M,g_M)$ into a Lorentzian manifold $(\mathcal{N},g_{\mathcal{N}})$. We can always assume that $M=D$ is a 2-disc with Euclidean metric when we encounter the regularity issue.  In the first subsection, we will rewrite the weakly Dirac-harmonic maps into certain ``standard form'', from which the continuity regularity of the map (i.e., \autoref{mthm:regularity}) is derived from \autoref{mthm:reg-anti-protential}, which is proved in \autoref{subsec:reg-anti-protential}.
\subsection{The local equations of weakly Dirac-harmonic map over 2-disc}
Note first that, if we set $\left\{ \nu_l \right\}_{l=n+1}^K$, $\nu_l=\sum_{a=1}^Kv_l^a\partial_{v^a}$, to be the orthonormal frame of $T^\perp N$ in $\mathbb{R}^K$ (note that, although $\bar{\iota}=\mathrm{id}\times\iota \mathpunct{:}\mathcal{N}\to \overline{\mathcal{N}}\mathpunct{:}=\mathbb{R}^1\times \mathbb{R}^K$ is a pseudo-Riemannian isometric embedding, $\left\{ \nu_l \right\}_{l=n+1}^K$ is not an orthonormal frame of $T^\perp \mathcal{N}$ in $T\overline{\mathcal{N}}$ in general, because the Riemannian metric of $N$ is not the restricted metric of $g_{\overline{\mathcal{N}}}$ on $N$.), then as for harmonic maps,
\[
  A(\nabla\varphi,\nabla\varphi)=\Theta\cdot\nabla\varphi,
\]
where, $\Theta=(\Theta^{da})$, and
\[
  \Theta^{da}=v_l^d(\varphi)\nabla v_l^a(\varphi)-v_l^a(\varphi)\nabla v_l^d(\varphi),
\]
which is clearly anti-symmetric.

Clearly, for a vector $T=(T^1,\ldots,T^K)\in TN\subset T \mathbb{R}^K$, the tangential part of $T$ can be written as
\[
  T^\top=(\widetilde{T}^1,\ldots,\widetilde{T}^K)=T-\left\langle T,\nu_l \right\rangle_{\mathbb{R}^K}\nu_l,
\]
i.e.,
\[
  \widetilde{T}^d=T^d-\sum_{e=1}^KT^e v_l^e v_l^d.
\]
In particular,
\[
  \begin{alignedat}{2}
    \overline{\mathcal{R}}^{\sharp\top}(\bar{\phi},\bar{\psi})&=(\widetilde{\mathcal{R}}^1,\ldots,\widetilde{\mathcal{R}}^K),&\quad
    \widetilde{\mathcal{R}}^d&=\frac{1}{2}\left( \bar{R}_{abcd}-\bar{R}_{abce} v_l^e v_l^d \right)\left\langle \bar{\psi}^a,\nabla\bar{\phi}^c\cdot\bar{\psi}^b \right\rangle_{\Sigma M},\\
    \overline{\mathcal{P}}^{\sharp\top}(\bar{\phi},\bar{\psi})&=(\widetilde{\mathcal{P}}^1,\ldots,\widetilde{\mathcal{P}}^K),&\quad
    \widetilde{\mathcal{P}}^d&=\left\langle \bar{A}_{ad}-\bar{A}_{ae} v_l^e v_l^d,\bar{A}_{bc} \right\rangle_{\overline{\mathcal{N}}}\left\langle \bar{\psi}^a,\nabla\bar{\phi}^c\cdot\bar{\psi}^b \right\rangle_{\Sigma M},\\
    \overline{\mathcal{H}}^{\sharp\top}(\bar{\phi},\bar{\psi})&=(\widetilde{\mathcal{H}}^1,\ldots,\widetilde{\mathcal{H}}^K),&\quad
    \widetilde{\mathcal{H}}^d&=\frac{1}{2}\left( \partial_d\lambda-\partial_e\lambda v_l^e v_l^d \right) \lvert V^\sharp \rvert^2 -\left( \vartheta^d-\vartheta^e v_l^e v_l^d \right)\div (\lambda V^\sharp)\\
                             &&&\qquad-\left[\partial_a\vartheta_d-\partial_d\vartheta_a-(\partial_a\vartheta_e-\partial_e\vartheta_a) v_l^e v_l^d\right]\left\langle \lambda V^\sharp,\nabla\varphi^a \right\rangle.
  \end{alignedat}
\]
If we set
\[
  \begin{alignedat}{2}
    \Upsilon_{db}&\mathpunct{:}=\partial_b\vartheta_d-\partial_d\vartheta_b-(\partial_b\vartheta_e-\partial_e\vartheta_b) v_l^e v_l^d,&\quad
    \Upsilon_d&\mathpunct{:}=\frac{1}{2\lambda^2(\varphi)}\left( \partial_d\lambda-\partial_e\lambda v^e_l v^d_l \right),\\
    \mathcal{Q}^d&\mathpunct{:}=\vartheta^d-\vartheta^e v_l^e v_l^d,&\quad
    \mathcal{RP}^d(\bar{\phi},\bar{\psi})&\mathpunct{:}=\widetilde{\mathcal{R}}^d+\widetilde{\mathcal{P}}^d,
  \end{alignedat}
\]
then the first equation in \eqref{eq:weak-Dirac-Harmonic} is
\begin{equation}\label{eq:div-vphi}
  \begin{split}
    -\div \nabla\varphi^d&=\Theta^{da}\nabla\varphi^a+\Upsilon_{da}\lambda V^\sharp\cdot\nabla\varphi^a-\Upsilon_d\lambda V^\sharp\cdot\lambda(\varphi)(\nabla\varphi^0+\vartheta_a\nabla\varphi^a)\\
                         &\qquad+\mathcal{RP}^d(\bar{\phi},\bar{\psi})+\mathcal{Q}^d \div (\lambda V^\sharp).
  \end{split}
\end{equation}
By \eqref{eq:bar-phi-extrinsic}, we know that the second equation in \eqref{eq:weak-Dirac-Harmonic} is
\begin{equation}\label{eq:div-lambda-V-sharp}
  -\div\left( \lambda(\varphi)(\nabla\varphi^0+\vartheta_a\nabla\varphi^a) \right)=\mathcal{RP}_0(\bar{\phi},\bar{\psi}),
\end{equation}
where
\[
  \mathcal{RP}_0(\bar{\phi},\bar{\psi}) \mathpunct{:}=\frac{1}{2}\bar{R}_{abc0}(\bar{\phi})\left\langle \bar{\psi}^a,\nabla\bar{\phi}^c\cdot\bar{\psi}^b \right\rangle_{\Sigma M}+\left\langle \bar{A}_{bc},\bar{A}_{a0} \right\rangle_{\bar{\phi}^{-1}T\overline{\mathcal{N}}}\left\langle \bar{\psi}^a,\nabla\bar{\phi}^c\cdot\bar{\psi}^b \right\rangle_{\Sigma M}.
\]
Finally, recall that the component equation of $\bar{\psi}=(\psi^0,\psi^1,\ldots,\psi^K)$ in \eqref{eq:weak-Dirac-Harmonic} is given by \eqref{eq:psi-extrinsic-compo}, and the twisted bundle $\Sigma D\otimes\bar{\phi}^{-1}T\overline{\mathcal{N}}$ is trivial, thus $\bar{\psi}$ can be viewed as a vector valued function from $D$ to $\mathbb{C}^2\otimes \mathbb{R}^{K+1}$. In conclusion, if we transform \eqref{eq:div-vphi}, \eqref{eq:div-lambda-V-sharp} and \eqref{eq:psi-extrinsic-compo} into matrix form, then we obtain the following proposition.
\begin{prop}\label{prop:weak-dirac-harmonic-local}
  Suppose $(\phi,\psi)\in \mathcal{X}^w(D,\mathcal{N})$ is a weakly Dirac-harmonic map. Then locally, under the identification of $\bar{\phi}$, $\bar{\psi}$ with $\phi=(\varphi^0,\varphi)=(\varphi^0,\varphi^1,\ldots,\varphi^K)$ and $\psi=(\psi^0,\psi^1,\ldots,\psi^K)$ respectively, the equation of $\psi$ and $\phi$ can be written as
  \begin{equation}\label{eq:psi-simple}
    \p\psi^d=B_b^d\cdot\psi^b,
  \end{equation}
  for some $B=(B^a_b)_{n\times n}\in L^2(D)$, $n=K+1$, and
  \begin{equation}\label{eq:phi-simple}
    -\div(Q\nabla\phi)=\Theta\cdot Q\nabla\phi+F\Omega\cdot Q\nabla\phi+\upsilon,
  \end{equation}
  respectively, where
  \[
    \begin{alignedat}{2}
      Q&=Q(\varphi)=
      \begin{pmatrix}
        \lambda(\varphi)&\lambda(\varphi)\vartheta\\
        0 & I_K
      \end{pmatrix}_{n\times n},&\quad\vartheta&=(\vartheta_1,\ldots,\vartheta_K),\\
      \Theta&=
      \begin{pmatrix}
        0&0\\
        0&\left( \Theta^{ab} \right)_{K\times K}
      \end{pmatrix}_{n\times n},&\quad
      F&=
      \begin{pmatrix}
        0&0&\cdots&0\\
        -\Upsilon_1&\Upsilon_{11}&\cdots&\Upsilon_{1K}\\
        \vdots  &             &\ddots&\vdots\\
        -\Upsilon_K&\Upsilon_{K1}&\cdots&\Upsilon_{KK}
      \end{pmatrix}_{n\times n},\\
      \Omega&=\lambda(\varphi)\mathrm{diag}(V^\sharp,\ldots,V^\sharp)_{n\times n},&\quad V^\sharp &\mathpunct{:}=\nabla\varphi^0+\vartheta_a\nabla\varphi^a,\\
      W&=\mathrm{diag}(w,\ldots,w)_{n\times n},&\quad w &\mathpunct{:}=\mathcal{RP}_0(\phi,\psi),\\
      \upsilon&=(w,v^1,\ldots, v^K)_{n\times 1}^T,&\quad v^d &\mathpunct{:}=\mathcal{RP}^d-\mathcal{Q}^d w.
    \end{alignedat}
  \]
\end{prop}
\subsection{An improved \texorpdfstring{$L^p$}{Lp}-regularity of the spinor} \label{sec:regularity-spinor}
Note that the terms $w$ and $\upsilon$ in \eqref{eq:phi-simple} are equivalent to $\left\langle \psi,\nabla\phi\cdot\psi \right\rangle_{\Sigma D}$, which are a priori in $L^1(D)$ merely. Our next step is to improve the regularity of $\psi$, which implies \eqref{eq:phi-simple} is an $L^q(D)$ perturbation of the equation for weakly harmonic maps into Lorentzian manifold for some $q$ with $1<q<2$.
\begin{thm}\label{thm:reg-spinor}
  Suppose $(\phi,\psi)\in \mathcal{X}^w(D,\mathcal{N})$ is a weakly Dirac-harmonic map from $D$ to $\mathcal{N}$, then $\psi$ solves \eqref{eq:psi-simple} weakly. Moreover, $\psi\in L^p(D)$ for any $p>4$.
\end{thm}
Since $\phi\in W^{1,2}$ and $\psi\in W^{1,4/3}$, by Sobolev embedding theorem, $\phi\in W^{1,2}\cap L^q$ for any $q\in(1,+\infty)$ and $\psi\in L^4$. Note that $\nabla\phi\in L^2$, $\bar{A}$ and $\bar{\Gamma}$ are $L^\infty$ bounded as remarked in \autoref{prop:iso-embedding}, if we take $D'\subset D$ small enough, the smallness condition is satisfied in the following lemma, which in return shows that $\psi\in L^p(D)$ for any $p>4$.

Although we only need to apply the following lemma to the case $m=2$, we state here the higher dimensional case, where the $L^p$ norm is replaced by Morrey norm. Such kind of result was first obtained in \cite{Wang2010remark} for some other Dirac type equations.
\begin{lem}[\cite{JostKeslerTolksdorfWuZhu2018Regularity}*{Lem.~2}]
  Suppose $m\geq2$, $B_1\subset \mathbb{R}^m$ and $\psi\in L^4(B_1,\mathbb{C}^L\otimes\mathbb{R}^{K+1})$ is a weak solution of the nonlinear system
  \[
    \p\psi^a=A^a_b\psi^b+B^a,\quad 0\leq a,b\leq K,
  \]
  where $A\in M_2^2(B_1,\mathfrak{gl}(L,\mathbb{C})\otimes \mathfrak{gl}(K+1,\mathbb{R}))$ and $B\in M_2^2(B_1,\mathbb{C}^L\otimes \mathbb{R}^{K+1})$. Then, for any $4<p<+\infty$, $\psi\in L^p(B_{1/2})$ provided that $\lVert A \rVert_{M_2^2(B_1)}\leq\epsilon_0$, for some constant $\epsilon_0=\epsilon_0(m,p)>0$.
\end{lem}
\subsection{The continuity of generalized harmonic maps}\label{subsec:reg-anti-protential}
In this subsection, we will prove \autoref{mthm:reg-anti-protential}, which is a further generalization of the one in \cite{Zhu2013Regularity}*{Thm.~1.2}, and initially obtained by Rivi\`ere and Struwe \cite{RiviereStruwe2008Partial}*{Thm.~1.1} for elliptical systems with anti-symmetric structure, and developed by Hajlasz, Strezelecki and Zhong \cite{HajaszStrzeleckiZhong2008new}*{Thm.~1.2} and Schikorra \cite{Schikorra2010remark}*{Rmk.~3.4}. The proof is similar to \cite{Zhu2013Regularity}*{Thm.~1.2}. The first step is to apply the Hodge decomposition (see \autoref{thm:hodge}) to $\Omega$ and obtain a divergence free structure with additional perturbation term, which needs to be handled carefully.
\begin{proof}[Proof of \autoref{mthm:reg-anti-protential}]
  By Hodge decomposition \autoref{thm:hodge}, there exist $\eta\in W^{1,2}(B,\mathrm{M}(n))$ and $\zeta\in W^{1,2}(B, \mathrm{M}(n)\otimes\wedge^{m-2}\mathbb{R}^m)$, such that
  \begin{equation}\label{eq:hodge}
    \Omega=\nabla\eta+\curl\zeta,\quad x\in B.
  \end{equation}
  Moreover,
  \begin{equation}\label{eq:hodge-est}
    \lVert \eta \rVert_{W^{1,2}(B_{2/3})}\leq C\lVert \Omega \rVert_{L^2(B)},\quad\lVert \nabla\zeta \rVert_{M_2^2(B)}\leq\lVert \Omega \rVert_{M_2^2(B)}.
  \end{equation}
  By \eqref{eq:thm-anti-upsilon} we know that $\Delta\eta=-W\in L^q(B)$, thus $\eta\in W^{2,q}(B)$ and for $q^*=mq/(m-q)$,
  \[
    \lVert \nabla\eta \rVert_{L^{q^*}(B_1/2)}\leq C\lVert \eta \rVert_{W^{2,q}(B_{1/2})}\leq C\left( \lVert \eta \rVert_{L^q(B_{2/3})}+\lVert W \rVert_{L^q(B_{2/3})} \right).
  \]

  Apply the above argument to the scaled equation
  \[
    -\div\widetilde{\Omega}=\widetilde{W},\quad \widetilde{\Omega}(x)=\Omega(x_0+Rx),\quad \widetilde{W}(x)=RW(x_0+Rx),\quad x\in B,
  \]
  we obtain
  \begin{equation}\label{eq:bar-upsilon}
    \lVert \nabla\eta \rVert_{L^{q^*}(B_{R/2}(x_0))}\leq CR^{m/q^*}\left( R^{-m/2}\lVert \Omega \rVert_{L^2(B_R(x_0))}+R^{1-m/q}\lVert W \rVert_{L^q(B_R(x_0))} \right).
  \end{equation}

  Note that \eqref{eq:thm-anti-u} implies
  \begin{equation}\label{eq:Q-u}
    -\div (Q\nabla u)=\Theta\cdot Q\nabla u+F \curl\zeta\cdot G\nabla u+F \mathcal{A}\cdot G\nabla u+\upsilon.
  \end{equation}
  where
  \[
    \mathcal{A}=\mathrm{diag}(\underbrace{\nabla\eta,\ldots,\nabla\eta}_n).
  \]
  Let $\bar{\upsilon} \mathpunct{:}=F \mathcal{A}\cdot G\nabla u$ and $\tilde{\upsilon} \mathpunct{:}=\bar{\upsilon}+\upsilon$. Note that, as we remarked our theorem holds also for $\upsilon\in L^s$ for some $s>m/2$, however, $\bar{\upsilon}\in L^q$ for some $1<q<2$ merely, which explains why we need to handle them differently. By \eqref{eq:thm-anti-condi}, if we take $\epsilon=\epsilon(m,\Lambda)>0$ small enough, then we can apply \autoref{lem:coulomb} to \eqref{eq:Q-u} to show, there exist $P\in W^{1,2}(B,\mathrm{SO}(n))$ and $\xi\in W_0^{1,2}(B,\mathfrak{so}(n)\otimes\wedge^{m-2}\mathbb{R}^m)$, such that
  \begin{equation}\label{eq:P-Q-u}
    -\div (P^{-1}Q\nabla u) =\curl\xi\cdot P^{-1}Q\nabla u+P^{-1}F \curl\zeta\cdot G\nabla u+P^{-1}\tilde{\upsilon}.
  \end{equation}
  If we write $P^{-1}=(P_{ab})$, $F=(F^{ab})$, $G=(G^{ab})$, $Q=(Q^{ab})$, $\Theta=(\Theta^{ab})$, $\zeta=(\zeta^{ab})$, $\xi=(\xi^{ab})$, then \eqref{eq:P-Q-u} can be written as
  \begin{equation}\label{eq:P-Q-u-component}
    -\div (P_{ab}Q^{bc}\nabla u^c)=P_{bc}Q^{cd}\curl\xi^{ab}\cdot\nabla u^d+P_{ab}F^{bc}G^{de}\cdot \curl\zeta^{cd}\nabla u^e+P_{ab}\tilde{\upsilon}^b.
  \end{equation}
  Since $P^{-1}\in W^{1,2}(B,\mathrm{SO}(n))$, $F\in W^{1,2}\cap L^\infty(B, \mathrm{M}(n))$, $G\in W^{1,2}\cap L^\infty(B,\mathrm{M}(n))$ and $Q\in W^{1,2}\cap L^\infty(B,\mathrm{GL}(n))$, we have $P_{bc}Q^{cd}\in W^{1,2}\cap L^\infty(B)$, $P_{ab}F^{bc}G^{de}\in W^{1,2}\cap L^\infty(B)$. Apply \eqref{eq:thm-anti-condi-infty}, it is easy to show
  \begin{equation}\label{eq:coeffi-est-Morrey}
    \begin{multlined}
      \lVert \nabla(P_{bc}Q^{cd}) \rVert_{M_2 ^2(B)}+\lVert \nabla(P_{ab}F^{bc}G^{de}) \rVert_{M_2 ^2(B)}\\
      \leq C(\Lambda)\left( \lVert \nabla P \rVert_{M_2 ^2(B)}+\lVert \nabla Q \rVert_{M_2 ^2(B)}+\lVert \nabla F \rVert_{M_2 ^2(B)}+\lVert \nabla G \rVert_{M_2 ^2(B)} \right).
    \end{multlined}
  \end{equation}
  Combining it with \eqref{eq:coulomb-est} and the assumption \eqref{eq:thm-anti-condi}, note also \eqref{eq:hodge-est}, we obtain
  \begin{equation}
    \begin{multlined}
      \lVert \nabla u \rVert_{M_2 ^2(B)}+\sum_{c}\lVert \nabla(P_{bc}Q^{cd}) \rVert_{M_2 ^2(B)}+\sum_{b,c}\lVert \nabla(P_{ab}F^{bc}G^{de}) \rVert_{M_2 ^2(B)}\\
      +\lVert \nabla\xi \rVert_{M_2 ^2(B)}+\lVert \curl\zeta \rVert_{M_2 ^2(B)} \leq C(\Lambda)\epsilon(m,\Lambda).
    \end{multlined}
  \end{equation}
  On the other hand, since $P^{-1}\in \mathrm{SO}(n)$, it follows from \eqref{eq:thm-anti-condi-infty} that
  \begin{equation}\label{eq:PiQ-nabla-u}
    \frac{1}{C(\Lambda)}\lvert \nabla u \rvert\leq \lvert P^{-1}Q\nabla u \rvert=\lvert Q\nabla u \rvert\leq C(\Lambda)\lvert \nabla u \rvert.
  \end{equation}

  Let $x_0\in B$, $0<r<R< \frac{1}{2}\mathrm{dist}(x_0,\partial B)$, and apply Hodge decomposition (see \cite{IwaniecMartin2001Geometric}*{Cor.~10.5.1}) to $P^{-1}Q\nabla u$, we can find $f\in W_0^{1,2}(B_R(x_0), \mathbb{R}^n)$, $g\in W_0^{1,2}(B_R(x_0), \mathbb{R}^n\otimes\wedge^{m-2}\mathbb{R}^m)$ and harmonic $h\in C^\infty(B_R(x_0), \mathbb{R}^n\otimes \mathbb{R}^m)$, such that
  \begin{equation}\label{eq:hodge-P-Q-u}
    P^{-1}Q\nabla u=\nabla f+\curl g+h,\quad \xtext{for a.e. $x\in B_R(x_0)$,}
  \end{equation}
  and by \eqref{eq:P-Q-u},
  \begin{equation}\label{eq:f}
    \begin{dcases}
      -\Delta f=\curl\xi\cdot P^{-1}Q\nabla u+P^{-1}F\curl\zeta\cdot G\nabla u+P^{-1}\tilde{\upsilon},& x\in B_R(x_0)\\
      f=0,& x\in\partial B_R(x_0),
    \end{dcases}
  \end{equation}
  and
  \begin{equation}\label{eq:g}
    \begin{dcases}
      -\Delta g=*\left(d(P^{-1}Q)\wedge d u\right),&x\in B_R(x_0)\\
      g=0,&x\in \partial B_R(x_0).
    \end{dcases}
  \end{equation}
  Fix $1<p<\frac{m}{m-1}$. Since $h$ is harmonic, we know that (see \cite{Giaquinta1983Multiple}*{Thm.~2.1}),
  \[
    \int_{B_r(x_0)}\lvert h \rvert^p\leq C(p)\left( \frac{r}{R} \right)^{m}\int_{B_R(x_0)}\lvert h \rvert^p.
  \]
  Then, by \eqref{eq:PiQ-nabla-u} and \eqref{eq:hodge-P-Q-u},
  \begin{equation}\label{eq:u-Br}
    \int_{B_r(x_0)}\lvert \nabla u \rvert^p\leq C(p,\Lambda)\left( \int_{B_R(x_0)}\lvert \nabla f \rvert^p+\int_{B_R(x_0)}\lvert \curl g \rvert^p+\left( \frac{r}{R} \right)^m\int_{B_R(x_0)}\lvert \nabla u \rvert^p \right).
  \end{equation}

  First, we estimate $\lVert \nabla f \rVert_{L^p(B_R(x_0))}$. Since $f=0$ on $\partial B_R(x_0)$, by duality,
  \[
    \lVert \nabla f \rVert_{L^p(B_R(x_0))}\leq C(p)\sup_{\substack{\varphi\in C_0^\infty(B_R(x_0))\\\lVert \varphi \rVert_{W^{1,p^*}}\leq 1}}\int_{B_R(x_0)}\nabla f\cdot\nabla\varphi,
  \]
  where $p^*=p/(p-1)$, and thereafter the norms refer to the domain $B_R(x_0)$. Note that $W^{1,p^*}_0(B_R(x_0))\hookrightarrow C^{1-m/p^*}(B_R(x_0))$. Therefore, for any $\varphi\in W^{1,p^*}_0(B_R(x_0))$, with $\lVert \varphi \rVert_{W^{1,p^*}}\leq 1$, we have
  \[
    \lVert \varphi \rVert_{L^\infty}\leq C R^{1-m/p^*}\lVert \varphi \rVert_{W^{1,p^*}}\leq CR^{1-m/p^*},\quad
    \lVert \nabla\varphi \rVert_{L^2}\leq CR^{m/2-m/p^*}.
  \]
  Moreover, by \eqref{eq:f}, we estimate
  \begin{align*}
    \int_{B_R(x_0)}\nabla f\cdot\nabla\varphi&=-\int_{B_R(x_0)}\Delta f\varphi\\
                                             &=\int_{B_R(x_0)} P_{bc}Q^{cd}\varphi^a\curl\xi^{ab}\cdot\nabla u^d+ P_{ab}F^{bc}G^{de}\varphi^a\curl\zeta^{cd}\cdot\nabla u^e +P_{ab}\tilde{\upsilon}^b\varphi^b\\
                                             &=\I+\II+\III.
  \end{align*}
  To simplify the notation in what follows, we also introduce the following notations:
  \[
    \mathcal{J}_p(x,r)\mathpunct{:}=\frac{1}{r^{m-p}}\int_{B_r(x)}\lvert \nabla u \rvert^p,\quad
    \mathcal{M}_p(y,R)\mathpunct{:}=\sup_{B_r(x)\subset B_R(y)}\mathcal{J}_p(x,r),\quad \mathcal{M}_p(R)=\mathcal{M}_p(0,R).
  \]
  By \autoref{lem:Hardy-BMO-ineq} and the conditions \eqref{eq:thm-anti-condi}, \eqref{eq:thm-anti-condi-infty}, we obtain,
  \[
    \I+\II=\int_{B_R(x_0)}P_{bc}Q^{cd}\varphi^a\curl\xi^{ab}\cdot\nabla u^d
    \leq C(\Lambda)\epsilon(m,\Lambda)R^{m-1-m/p^*}\mathcal{M}_p(x_0,2R),
  \]

  Now, by assumption, $\Omega\in M_2 ^2(B)$ and $W\in M_2 ^q(B)$. We see that from \eqref{eq:bar-upsilon}, for any $B_{2R}(x_0)\subset B$,
  \[
    \lVert \nabla\eta \rVert_{L^{q^*}(B_{R}(x_0))}\leq C R^{m/q^*-1}\left( \lVert \Omega \rVert_{M_2 ^2(B_{2R}(x_0))}+R^{2-2/q}\lVert W \rVert_{M_2 ^q(B_{2R}(x_0))} \right).
  \]
  By H\"older's inequality, for $\bar{q}=1/(1/q+1/p-1/m)$,
  \begin{align*}
    \lVert \bar{\upsilon} \rVert_{L^{\bar{q}}(B_R(x_0))}&\leq C(\Lambda)\lVert \mathcal{A} \rVert_{L^{q^*}(B_R(x_0))}\lVert \nabla u \rVert_{L^p(B_R(x_0))}\\
                                                        &\leq C(\Lambda)R^{(m-p)/p+m/q-2}\mathcal{M}_p(x_0,2R)\\
                                                        &\qquad\cdot \left( \lVert \Omega \rVert_{M_2 ^2(B_{2R}(x_0))}+R^{2-2/q}\lVert W \rVert_{M_2 ^q(B_{2R}(x_0))} \right),
  \end{align*}
  and
  \begin{align*}
    \III&=\int_{B_R(x_0)}P_{ab}\tilde{\upsilon}^b\varphi^b\leq C\lVert \tilde{\upsilon} \rVert_{L^1(B_R(x_0))}\lVert \varphi \rVert_{L^\infty(B_R(x_0))}\\
        &\leq C \left( R^{m-m/\bar{q}}\lVert \bar{\upsilon} \rVert_{L^{\bar{q}}(B_R(x_0))}+ R^{m-2/s}\lVert \upsilon \rVert_{M_2 ^s(B_R(x_0))} \right)\lVert \varphi \rVert_{L^\infty(B_R(x_0))}\\
        &\leq C(\Lambda)\epsilon(m,\Lambda)R^{m/p-1}\mathcal{M}_p(x_0,2R)\left(1+R^{2-2/q}\right)\\
        &\qquad+C(\Lambda)R^{m/p-1}R^{2-2/s}\lVert \upsilon \rVert_{M_2 ^s(B_{2R}(x_0))}.
  \end{align*}

  In conclusion,
  \begin{equation}\label{eq:grad-f}
    \begin{split}
      \lVert \nabla f \rVert_{L^p(B_R(x_0))}&\leq C(p,m,\Lambda) \epsilon(m,\Lambda)R^{m/p-1}\mathcal{M}_p(x_0,2R)(1+R^{2-2/q})\\
                                            &\qquad+C(p,m,\Lambda)R^{m/p-1}R^{2-2/s}\lVert \upsilon \rVert_{M_2 ^s(B_{2R}(x_0))}.
    \end{split}
  \end{equation}
  Similarly, \begin{equation}\label{eq:curl-g}
    \lVert \curl g \rVert_{L^p(B_R(x_0))}\leq C(\Lambda)\epsilon(m,\Lambda)R^{m/p-1}\mathcal{M}_p(x_0,2R).
  \end{equation}
  Combining \eqref{eq:grad-f}, \eqref{eq:curl-g} and \eqref{eq:u-Br}, we obtain that, for $\delta=2-2/s>0$, $\delta'=2-2/q>0$,
  \begin{align*}
    \frac{1}{r^{m-p}}\int_{B_r(x_0)}\lvert \nabla u \rvert^p
    &\leq C(p,\Lambda,m)\biggl[ \left(\frac{R}{r}\right)^{m-p} \epsilon(m,\Lambda)^p\mathcal{M}_p^p(x_0,2R)\left(1+R^{\delta' p}\right) \\
    &\myquad[7] +\left(\frac{R}{r}\right)^{m-p}R^{\delta p}\lVert \upsilon \rVert_{M^s_{m-2}(B_{2R}(x_0))}^p\\
    &\myquad[7] +\left( \frac{r}{R} \right)^p \frac{1}{R^{m-p}}\int_{B_R(x_0)}\lvert \nabla u \rvert^p \biggr].
  \end{align*}

  For some fixed $\gamma\in(0,1)$ with $C(p,\Lambda,m)\gamma^{(p-1)/2}\leq 1/6$ and $(\gamma/2)^{\delta p}<1/4$, we can choose $\epsilon(m,\Lambda)$ small enough, such that  $\epsilon(m,\Lambda)\leq\gamma^{m/p}$. Now, take $r=\gamma R$, the above inequality implies,
  \begin{align*}
    \mathcal{J}_p(x_0,\gamma R)&\leq C(p,\Lambda,m) \gamma^{p-m}\Bigl( \gamma^m \mathcal{M}_p^p(x_0,2R)(2+R^{\delta' p})+R^{\delta p}\lVert \upsilon \rVert_{M^s_{m-2}(B_{2R}(x_0))} \Bigr) \\
                               &\leq \frac{1}{2}\mathcal{M}_p^p(x_0,2R)+C(p,\Lambda,m)\gamma^{p-m}R^{\delta p}\lVert \upsilon \rVert_{M^s_{m-2}(B_{2R}(x_0))}.
  \end{align*}
  Since the above inequality is valid for any $B_{2R}(x_0)\subset B$ and $r<R$, in particular, for any fixed $R'\in(0,1]$, we can pass to the supremum with respect to $B_{2R}(x_0)\subset B_{R'}\subset B$ to obtain (note that $B_{\gamma R}(x_0)\subset B_{\gamma R'/2}$),
  \[
    \begin{split}
      \mathcal{M}_p(\gamma R'/2)&\leq \frac{1}{2}\mathcal{M}_p(R')+C(p,\Lambda,m)\gamma^{p-m}R'^{\delta p}\lVert \upsilon \rVert_{M_2 ^s(B_{R'})}^p\\
                                &\leq \frac{1}{2}\mathcal{M}_p(R')+C(p,\Lambda,m)\gamma^{p-m}R'^{\delta p}\lVert \upsilon \rVert_{M_2 ^s(B)}^p.
    \end{split}
  \]
  Let $\lambda=\gamma/2$ and $\alpha$ satisfies $\lambda^{p\alpha}=1/2$, i.e., $\alpha=[\ln_{\gamma/2}(1/2)]/p\in(0,1)$, then
  \begin{equation}\label{eq:iterate}
    \mathcal{M}_p(\lambda R')\leq \lambda^{p\alpha} \mathcal{M}_p(R')+2^{p-m}C(p,\Lambda,m)\lambda^{p-m}R'^{\delta p}\lVert \upsilon \rVert_{M_2 ^s(B)}^p.
  \end{equation}

  Now, we iterate \eqref{eq:iterate} as follows: for any given $r\in(0,\lambda)$, suppose $\lambda^{l+1}<r\leq\lambda^l$ for some $l\in \mathbb{N}$, and we denote $C \mathpunct{:}=C(p,\Lambda,m)2^{p-m}\lVert \upsilon \rVert^p_{M_2 ^s(B)}$ for simplicity, then since $\lambda^{\delta p}<1/4$ and $\lambda^{p\alpha}=1/2$, we know that $\lambda^{(\alpha-\delta)p}>2$, and
  \begin{align*}
    \mathcal{M}_p(r)&\leq \mathcal{M}_p(\lambda^l)\leq \lambda^{p\alpha} \mathcal{M}_p(\lambda^{l-1})+C\lambda^{p-m+(l-1)\delta p}\\
                    &\leq \lambda^{lp\alpha}\mathcal{M}_p(1)+C\lambda^{p-m}\sum_{i=1}^l\lambda^{(i-1)p\alpha+(l-i)\delta p}\\
                    &\leq\lambda^{lp\alpha}\mathcal{M}_p(1)+C\lambda^{p-m-\delta p}\cdot\lambda^{lp\alpha}\\
                    &\leq2\left(\epsilon(m,\lambda)^p+C2^{(m+\delta p-p)/p\alpha}\right)r^{p\alpha},
  \end{align*}
  because $\mathcal{M}_p(1)\leq \mathcal{M}_2(1)<\epsilon(m,\Lambda)^p$. The required estimate in \autoref{mthm:reg-anti-protential} follows from the characteristic of H\"older continuity by Dirichlet growth (see \cite{Giaquinta1983Multiple}*{Chap.~III, Thm.~1.1}).
\end{proof}

\begin{rmk}
  Similarly to the observation as in \cite{Rupflin2008improved}*{Prop.~2.1}, the same conclusion as in \autoref{mthm:reg-anti-protential} holds if we replace $\lVert \upsilon \rVert_{M_{m-2}^s(B)}$ by $\lVert \upsilon \rVert_{L^s(B)}$ and require that $s>m/2$. See also \cite{Sharp2014Higher}.
\end{rmk}

Now, we are ready to prove \autoref{mthm:regularity}. For $m=2$, $n=K+1$, by \autoref{prop:weak-dirac-harmonic-local}, the equation of $\phi$ is given by \eqref{eq:phi-simple}, which has exactly the same form of \eqref{eq:thm-anti-u}. However, we need to verify the conditions in \autoref{mthm:reg-anti-protential}. Note that, \eqref{eq:thm-anti-upsilon} is just $-\div(\lambda(\varphi)V^\sharp)=w$, which is included in the equation \eqref{eq:phi-simple}. By \autoref{thm:reg-spinor}, $\psi\in L^p(D)$ for any $p>4$. Therefore, $w\in L^q(D)$ for any $1<q<2$. Similarly, $\upsilon\in L^q(D)$ for any $1<q<2$. Here, we need the remark after \autoref{prop:iso-embedding} to show the $L^\infty$-boundedness of the components for Christoffel symbols, the pseudo-Riemannian curvature and the second fundamental form. The smallness condition \eqref{eq:thm-anti-condi} is satisfied provided we take $B=B_{4R}(x_0)$ small enough. The rest conditions in \autoref{mthm:reg-anti-protential} are easy to verify, and it implies that $\phi$ is H\"older continuous in $B_{2R}(x_0)$. By the arbitrariness of $B$, we show the H\"older continuous of $\phi$ over the Riemann surface $M$. This finishes the proof of \autoref{mthm:regularity}.
\section{The smoothness of weakly Dirac-harmonic maps}\label{sec:smooth}
The main content of this section is devoted to improving the regularity of weakly Dirac-harmonic maps. Recall that for weakly Dirac-harmonic map $(\phi,\psi)\in \mathcal{X}^w$, by \autoref{mthm:regularity}, we have already shown that $\phi$ is H\"older continuous. In that case, we can write the Euler--Lagrange equation into \eqref{eq:EL-smooth}. The main obstruction to apply bootstrap argument to the equation \eqref{eq:EL-smooth} of $(\phi,\psi)$ is the $C^{1,\alpha}$-regularity of $\phi$. Since the following argument holds for general pseudo-Riemannian target manifold (not only Lorentzian manifold), we will prove \autoref{mthm:smooth} together.

For the case of Riemannian target, it was proved in \cite{ChenJostLiWang2005Regularity}*{Thm.~2.3} that such $C^{1,\alpha}$-regularity for Dirac-harmonic maps hold. They follow closely to \cite{Jost2011Riemannian}*{Sect.~8.4}, where the general and classical regularity theorem of Ladyzhenskaia--Ural'tzeva~\cite{LadyzhenskaiaUralctzeva1961smoothness}*{Lem.~2} and Morrey~\cite{Morrey1966Multiple}*{Lem.~5.9.1} is applied to the harmonic maps equations. We summarize \cite{ChenJostLiWang2005Regularity}*{Thm.~2.3} into the following abstract form, which has been generalized to the pseudo-Riemannian target effortlessly.
\begin{thm}\label{thm:C-alpha-reg}
  Suppose $\phi$ is a \emph{continuous} map from a disc $D=B_{2R}(x_0)\subset \mathbb{R}^2$ to a pseudo-Riemannian manifold $\mathcal{N} \hookrightarrow \mathbb{R}^K$ and $\psi$ is a $W^{1,4/3}$ section of the twisted spin bundle $\Sigma D\otimes\phi^{-1}T \mathcal{N}$,  satisfy the following elliptic system
  \begin{equation}\label{eq:Dirac-harmonic-sys}
    \begin{cases}
      \Delta\phi^i=-G^i(x,\phi,\psi, d\phi)\\
      \p\psi^i=H^i_k(x,\phi,d\phi)\cdot\psi^k,
    \end{cases}
  \end{equation}
  with $G=(G^1,\ldots,G^K),H=(H_k^i)_{i,k=1}^K$ satisfies the following conditions over $D$,
  \begin{equation}\label{eq:Dirac-harmonic-sys-condi}
    \begin{alignedat}{2}
      \lvert G \rvert                 &\leq C\left( \lvert d\phi \rvert^2+\lvert d\phi \rvert\lvert \psi \rvert^2 \right),&\quad
      \lvert \partial_xG \rvert       &\leq C(\lvert d\phi \rvert^3+\lvert \nabla\psi \rvert\lvert \psi \rvert\lvert d\phi \rvert),\\
      \lvert \partial_{\phi}G \rvert  &\leq C(\lvert d\phi \rvert^2+\lvert \psi \rvert^2\lvert d\phi \rvert),&\quad
      \lvert \partial_{d\phi}G \rvert &\leq C(\lvert d\phi \rvert+\lvert \psi \rvert^2),\\
      \lvert H \rvert                 &\leq C\lvert \phi(x)-\phi(x_0) \rvert\lvert d\phi \rvert.
    \end{alignedat}
  \end{equation}
  Then $\phi\in C^{1,\alpha}(B_R(x_0))$ and $\psi\in C^\alpha(B_R(x_0))$ for any $\alpha\in(0,1)$, provided that $R$ is sufficiently small.
\end{thm}
\begin{proof}[Sketch of the proof]
  The idea is to show first that $\phi\in W^{2,2}\cap W^{1,4}(B_R(x_0),\mathcal{N})$, which is based on the relation of weak derivatives and difference quotients, i.e., for $\phi\in C^0\cap W^{1,4}\cap W^{3,2}(B_{2R}(x_0),\mathcal{N})$, we can prove for small enough $R$,
  \begin{equation}\label{eq:phi-w22}
    \lVert \nabla^2\phi \rVert_{L^2(B_R(x_0))}+\lVert d\phi \rVert_{L^4(B_R(x_0))}^2\leq C\lVert d\phi \rVert_{L^2(B_{2R}(x_0))},
  \end{equation}
  and then replace the weak derivatives by difference quotients of $\phi$.

  Whenever we have shown $\phi\in W^{2,2}(B_R(x_0))\subset W^{1,p}(B_R(x_0))$ for any $p\geq1$, note the continuity of $\phi$, we know that the right-hand side equation of $\psi$ in \eqref{eq:Dirac-harmonic-sys} is in $L^p(B_R(x_0))$ for any $p>2$, and the $L^p$ estimates of Dirac operator (see \cite{ChenJostLiWang2006Diracharmonic}*{Lem.~4.7}) implies that $\psi\in C^\alpha(B_R(x_0))$ for any $\alpha>0$. The $L^p$ estimates for the equation of $\phi$ in \eqref{eq:Dirac-harmonic-sys} implies that $\phi\in W^{2,p}(B_R(x_0))$ for any $p>2$, and so $\phi\in C^{1,\alpha}(B_R(x_0))$.
\end{proof}
Now, since $\phi$ is continuous, we can choose local coordinates on $\mathcal{N}$, such that $\Gamma_{ij}^k(\phi(x_0))=0$, for all $i,j,k=0,1,\ldots,n$. Then it is easy to verify that \eqref{eq:EL-smooth} can be rewritten into the form of \eqref{eq:Dirac-harmonic-sys}, and the coefficients satisfies the conditions \eqref{eq:Dirac-harmonic-sys-condi}. Therefore, \autoref{thm:C-alpha-reg} implies that $\phi\in C^{1,\alpha}(B_R(x_0))$ and $\psi\in C^\alpha(B_R(x_0))$ for any sufficiently small $B_R(x_0)\subset M$.  By the elliptic estimates for the Dirac operator, we have $\psi\in C^{1,\alpha}(B_R(x_0))$. \autoref{mthm:smooth} follows from the standard bootstrap argument of elliptic theory and the arbitrariness of $x_0$.

\

\appendix
\section{Hodge decomposition, Coulomb gauge of Morrey type and Hardy--BMO duality}\label{sec:appx}
In this appendix, we state some classical results which are needed in the proof of \autoref{mthm:reg-anti-protential}. The first one is the following Sobolev-type Hodge decomposition theorem.
\begin{thm}[\cite{Bethuel1993singular}*{Prop.~I\!I.1}]\label{thm:hodge}
  Suppose $1<p<+\infty$ and $\omega \in W^{l,p}$ is a $k$-form on $\mathbb{R}^n$, then there is a $k-1$-form $\alpha\in W^{l+1,p}$ and a $k+1$-form $\beta\in W^{l+1,p}$, such that
  \[
    \omega=d\alpha+d^*\beta,\quad d^*\alpha=0=d\beta,
  \]
  and
  \[
    \lVert \alpha \rVert_{W^{l+1,p}}+\lVert \beta \rVert_{W^{l+1,p}}\leq C(k,p)\lVert \omega \rVert_{W^{l,p}}.
  \]
  Moreover, $\alpha$ and $\beta$ are unique. If $d\omega=0$ (resp. $d^*\omega=0$), then $\beta=0$ (resp. $\alpha=0$).
\end{thm}
As a corollary, if we take a cutoff function $\rho\in C_0^\infty(B_2)$, with
\[
  \rho|_{B_1}\equiv1,\quad0\leq\rho\leq1,\quad\lvert \nabla\rho \rvert\leq 2/\rho,
\]
and apply \autoref{thm:hodge} to $\rho\omega$, then we obtain $\alpha,\beta\in W^{l+1,p}$, such that
\[
  \omega=d\alpha+d^*\beta,\quad d^*\alpha=0=d\beta,\quad x\in B_1
\]
and
\[
  \lVert \alpha \rVert_{W^{l+1,p}(B_1)}+\lVert \beta \rVert_{W^{l+1,p}(B_1)}\leq C(k,p)\lVert \omega \rVert_{W^{l,p}(B_2)}.
\]

The following lemma is a consequence of the existence of Uhlenbeck's Coulomb gauge (see \cite{RiviereStruwe2008Partial}*{Lem.~3.1}) and Hodge decomposition (see \cite{IqaniecMartin1993Quasiregular}*{Thm.~6.1} and \cite{Bethuel1993singular}*{Prop.~I\!I.1}).
\begin{lem}[\cite{RiviereStruwe2008Partial}*{Lem.~3.1}]\label{lem:coulomb}
  For every $m\in \mathbb{N}$, there exists $\epsilon=\epsilon(m)$, such that for every $\Theta\in L^2(B, \mathfrak{so}(n)\otimes\wedge^1 \mathbb{R}^m)$, $B\subset \mathbb{R}^m$, if $\lVert \Theta \rVert_{M_{2}^2}<\epsilon$, then one can find $P\in W^{1,2}(B;\mathrm{SO}(n))$ and $\xi\in W^{1,2}_0(B,\mathfrak{so}(n)\otimes\wedge^{m-2}\mathbb{R}^m)$ such that
  \begin{equation}\label{eq:coulomb}
    \begin{split}
      P^{-1}\nabla P+P^{-1}\Theta P&=\curl\xi,\quad x\in B,\\
      d\xi&=0,\quad x\in B,\\
      \xi|_{\partial B}&=0,
    \end{split}
  \end{equation}
  with the following estimate holds,
  \begin{equation}\label{eq:coulomb-est}
    \lVert \nabla P \rVert_{M_{2}^2(B)}+\lVert \nabla\xi \rVert_{M_{2}^2(B)}\leq C\lVert \Theta \rVert_{M_{2}^2(B)}.
  \end{equation}
\end{lem}
Recall that the BMO norm is defined as
\[
  \lVert f \rVert_{\mathrm{BMO}}\mathpunct{:}=\sup_{B_r(x)\subset \mathbb{R}^m}\frac{1}{\lvert B_r(x) \rvert}\int_{B_r(x)}\lvert f-f_{x,r} \rvert,
\]
where $f_{x,r}$ is the integral mean over $B_r(x)$, and the norm on Hardy space $\mathcal{H}^1$ is given by
\[
  \lVert f \rVert_{\mathcal{H}^1}\mathpunct{:}=\lVert f \rVert_{L^1}+\lVert f^* \rVert_{L^1},
\]
where
\[
  f^*(x)\mathpunct{:}=\sup_{r>0}\left\lvert \frac{1}{r^m}\int_{\mathbb{R}^m}f(y)\phi\left(\frac{x-y}{r}\right)dy \right\rvert,\quad\forall\phi\in C_0^\infty(\mathbb{R}^m),\quad \int_{\mathbb{R}^m}\phi=1.
\]

The key estimate in the proof of \autoref{mthm:reg-anti-protential} is given by the following lemma, which is usually referred to as Hardy-BMO duality. The following form is due to Fefferman \cite{Fefferman1971Characterizations} and Evans \cite{Evans1991Partial}, we also refer to Bethuel \cite{Bethuel1993singular}*{Prop.~I\!I\!I.2} for a proof.
\begin{lem}\label{lem:Hardy-BMO-ineq}
  Suppose $m\geq2$, $1\leq s<\infty$ and $1<p<\infty$, $p^*=p/(p-1)$. For any ball $B_R(x_0)\subset \mathbb{R}^m$, $f\in W^{1,p}(B_R(x_0))$, $g\in W^{1,p^*}(B_R(x_0),\wedge^{m-2} \mathbb{R}^m)$ and $h\in W^{1,s}(B_{2R}(x_0))$, if
  \[
    f|_{\partial B_R(x_0)}=0\quad\text{or}\quad g|_{\partial B_R(x_0)}=0,
  \]
  and
  \[
    \lVert \nabla h \rVert_{M_{s}^s(B_{2R}(x_0))}<+\infty,
  \]
  then
  \[
    \int_{B_R(x_0)}f\curl g\cdot\nabla h\leq C(m,s,p)\lVert \nabla f \rVert_{L^p(B_R(x_0))}\lVert \curl g \rVert_{L^{p^*}(B_R(x_0))}\lVert \nabla h \rVert_{M_{s}^s(B_{2R}(x_0))}.
  \]
\end{lem}

\end{document}